\documentclass[10pt]{article}
\usepackage[margin=1in]{geometry}
\usepackage[all]{xy}
\usepackage{graphicx}

\usepackage{amsmath,amsthm,amssymb,color,latexsym}
\usepackage{geometry} 
\usepackage{mathtools}       
\usepackage{graphicx}
\usepackage{bm}
\usepackage{youngtab}
\usepackage{ytableau,tikz,varwidth}
\ytableausetup{mathmode, boxsize=.17in}
\usetikzlibrary{arrows.meta}
\usepackage{mathrsfs}
\usepackage{hyperref}
\usepackage{comment}

\newtheorem{theorem}{Theorem}

\newtheorem{corollary}[theorem]{Corollary}
\newtheorem{lemma}[theorem]{Lemma}

\theoremstyle{definition}

\newtheorem{definition}[theorem]{Definition}
\newtheorem{remark}[theorem]{Remark}
\newtheorem{example}[theorem]{Example}
\newtheorem{proposition}[theorem]{Proposition}
\numberwithin{theorem}{section}

\newcommand{\slideEq}{=_{\text{slide}}}
\newcommand{\alg}{\mathrm{pesh}}
\newcommand{\algorithm}{partial evacuation shuffling}
\newcommand{\coplacticEvac}{e}

\newcommand{\T}{T}

\newcommand{\MoNBar}{\overline{M}_{0,r}}
\newcommand{\shuffle}{\mathrm{switch}}
\newcommand{\hop}{\mathrm{hop}}
\newcommand{\crystal}{\mathrm{crystal}}
\newcommand{\simpalg}{\mathrm{alg}}
\newcommand{\esh}{\mathrm{esh}}
\newcommand{\weight}{\mathrm{wt}}
\newcommand{\shape}{\mathrm{sh}}
\newcommand{\coswitch}{\mathrm{coswitch}}

\newcommand{\ybox}{\text{\scriptsize \yng(1)}}

\newcommand{\SSYT}{\mathrm{SSYT}}
\newcommand{\LR}{\mathrm{LR}}

\newcommand{\DE}{\mathrm{DE}}

\title{Local algorithms for coplactic switching and evacuation of Young tableaux}
\author{Kelsey M. Brown and Derek Moran}
\date{}

\begin{document}

\maketitle

\begin{abstract}
    Tableau switching is a well studied bijection on pairs of skew Young tableaux which swaps their relative positions. This is achieved by successively sliding the entries of the inner tableaux through the outer one via jeu de taquin (JDT) slides. Tableau coswitching is a similar but \emph{coplactic} operation, meaning it commutes with any sequence of JDT slides. Coswitching is defined by first performing JDT rectification on the union of the tableaux, switching the resulting pair, then unrectifying the union. This definition requires us to perform large scale modifications to the skew shapes during the rectification and unrectification steps, which is both computationally taxing and obscures the effect of coswitching on the pair of skew tableaux. In \cite{1box_a}, Gillespie and Levinson define an algorithm which computes coswitching as a sequence of local moves, which do not alter the skew shapes of the tableaux, when one of the tableaux is a single box. 
    
    In this paper, we extend the results of Gillespie and Levinson (and Gillespie, Levinson, and Purbhoo for type B) to the case the two skew tableaux are of arbitrary size. We also describe multiple bijections on tableaux, each descending to the \emph{evacuation shuffling} (esh) operation on pairs of dual equivalence classes. We show how special cases of our local algorithm compute the Sch\"utzenberger involution on skew tableaux, as well as the wall-crossing and monodromy of certain covering spaces of the moduli space $\overline{M}_{0, r}(\mathbb{R})$. 

\end{abstract}

\section{Introduction}
An operation on Young tableaux is said to be \emph{coplactic} if it commutes with every sequence of jeu de taquin (JDT) slides. In this paper, we give a combinatorial rule for the coplactic analog of \emph{tableau switching} on pairs of skew semistandard Young tableaux. We call this operation \emph{tableau coswitching}; it is defined by first performing JDT rectification on the union of the tableaux, switching the resulting pair, then ``unrectifying'' the union (a precise definition is given below, see Def. \ref{def:coswitch}).

Our work is motivated by certain covering spaces that arise in real Schubert calculus, which we now briefly describe. Consider the \textit{rational normal curve}, which is the image of the Veronese embedding $\mathbb{P}^1\hookrightarrow\mathbb{P}^{n-1}=\mathbb{P}(\mathbb{C}^n)$, defined by $t\mapsto[1:t:\ldots:t^{n-1}]$. The \textit{maximally tangent} or \textit{osculating flag} with respect to this curve at a point $t$ is the complete flag $\mathscr{F}_t$ in $\mathbb{C}^n$ given by the row spans of the matrix of iterated derivatives
$\begin{bmatrix}
    (\frac{d}{dt})^{i-1}(t^{j-1})
\end{bmatrix}_{0 \leq i, j \leq n-1}$.
%
%
Given a choice of partition $\lambda$ and a point $t$, we obtain a Schubert variety $\Omega_\lambda(\mathscr{F}_t)$ in the Grassmannian $\mathrm{Gr}(k,\mathbb{C}^n)$. We consider the intersection
\begin{equation} \label{eqn:schubert-intersection}
S=S(\lambda^{(1)},\ldots,\lambda^{(r)})=\Omega_{\lambda^{(1)}}(\mathscr{F}_{t_1})\cap\ldots\cap\Omega_{\lambda^{(r)}}(\mathscr{F}_{t_r}),
\end{equation}
where the osculation points $t_i$ are distinct and $\lambda^{(1)},\ldots,\lambda^{(r)}$ are partitions whose sizes sum to $k(n-k)$.

The Mukhin--Tarasov--Varchenko (MTV) theorem \cite{MTV}, previously the Shapiro--Shapiro conjecture, guarantees that when $t_1, \ldots, t_r$ are all real, the intersection \eqref{eqn:schubert-intersection} consists entirely of distinct real points, each occurring with multiplicity $1$. This is a surprising result, as in general, a system of real polynomial equations need not have all real solutions. 
Work of Purbhoo \cite{purbhoo2009jeutaquinmonodromyproblem} and Speyer \cite{speyer} shows that, as the osculation points vary, the intersections \eqref{eqn:schubert-intersection} extend to a flat family $\mathcal{S} = \mathcal{S}(\lambda^{(1)}, \ldots, \lambda^{(r)})$ over the moduli space $\MoNBar$ of stable curves of genus $0$ with $r$ marked points $t_1, \ldots, t_r$. The real points $\mathcal{S}(\mathbb{R}) \to \MoNBar(\mathbb{R})$ moreover form a smooth covering space, of degree equal to the Littlewood--Richardson (LR) coefficient. Geometrically, the covering space can be viewed as the moduli space parameterizing the solution sets guaranteed by the MTV theorem, extended to the closure in which two or more of the chosen osculating flags collide. There has been a lot of ongoing and recent followup to the MTV theorem; see \cite{brazelton2022enricheddegreewronski, Karp_2023, karp2023universalpluckercoordinateswronski, betterproof}. 

%
The monodromy of the covering $\mathcal{S}(\mathbb{R}) \to M_{0, r}(\mathbb{R})$ can be described entirely in terms of tableau combinatorics, either using growth diagrams \cite{speyer} or skew tableaux \cite{Levinson_2017}. It uses Haiman's dual equivalence classes of skew tableaux \cite{dual_equiv}. 
The core operation is the {\bf evacuation shuffling} ($\esh$) bijection on pairs of dual equivalence classes of skew tableaux $(X, T)$; see Def. \ref{def:evacuation-shuffling}. Whenever we write $(X, T)$, we assume the shape of $T$ extends that of $X$. The definition of $\esh$ involves large-scale modifications to the skew shapes of $X$ and $T$ via JDT rectification and ``un-rectification''. As such, computing $\esh$ by the definition is cumbersome and lengthy to compute in practice and obscures many structural properties of the bijection.

Our main result is a pair of algorithms to compute $\esh$ `locally', that is, without changing the skew shapes. Let $(X, T)$ be a pair of skew tableaux, with $X$ arbitrary and $T$ Littlewood-Richardson (see Sec. \ref{sec: LR-tab}), such that the shape of $T$ extends that of $X$. First, the \emph{hopping algorithm} (Def. \ref{def:type_a_hop}) computes $\esh(X, T)$ by a sequence of local moves resembling JDT slides. Second, the \emph{crystal algorithm} (Def. \ref{def:type-A-crystal-alg}) expresses $\esh(X, T)$ as a composition of Type A crystal operators.


In fact, the operation $\esh$ on dual equivalence classes lifts to several different natural bijections on skew Young tableaux, including \textbf{coswitching} as defined above, \textbf{\algorithm}, and the \textbf{Sch\"utzenberger involution} or \textbf{coplactic extension of evacuation}, which we define in detail below. The last is of particular interest because it has applications to the representation theory of Lie algebras \cite{Fulton}, to Kashiwara crystals \cite{lenart2006combinatoricscrystalgraphsi}, and to cactus groups \cite{halacheva2016alexandertypeinvariantstangles, henriques2005crystalscoboundarycategories, rodrigues2020actioncactusgroupshifted}.
Our algorithms compute each of these bijections.
\begin{theorem}[Theorems \ref{thm: hop=esh} and \ref{thm: hop=crystal}]
    Our algorithms compute coswitching and \algorithm{} on pairs $(X, T)$ with $X$ arbitrary and $T$ LR, and the Sch\"{u}tzenberger involution on arbitrary skew tableaux, in terms of local moves.
\end{theorem}

The case where $X$ is a single box was previously computed by Gillespie--Levinson \cite{1box_a}. Gillespie, Levinson, and Purbhoo later gave an analogous algorithm in type B, for skew shifted tableaux; see \cite{1box_b}. Our results similarly extend to type B using either local moves or the crystal-like coplactic operators on skew shifted tableaux defined in \cite{typeb_crystals}; see Def. \ref{def:type-b-hop} and \ref{def:type-b-crystal} and Theorems \ref{thm:type-b-hop=coswitch} and \ref{thm:type-b-hop=crystal}.

Our algorithms improve the complexity of the calculation of coswitching two tableaux from a maximum of approximately $2|\alpha|(l(\beta)+l(\lambda))+|\beta|l(\lambda)$ steps to a maximum of approximately $|\beta|(l(\beta)+|\beta|)$ steps, as shown in Section \ref{sec:complexity}.

\subsection{Outline}

In Section \ref{sec:background}, we start by reviewing the combinatorial background required to define and prove our algorithms, with the goal of working up to Littlewood-Richardson tableaux and dual equivalence. We then briefly describe the relevant coplactic bijections and their relation to the topology of $\MoNBar$. In Section \ref{sec:main results}, we define two local algorithms and prove they agree with coswitching. We then describe multiple variations and applications of the local algorithms in Section \ref{sec:related-results}.
In Section \ref{sec:type B}, we briefly review the background of shifted tableau combinatorics before defining the analogous local operations in type B.

\subsection{Acknowledgements}

We would like to thank our advisors, Maria Gillespie and Jake Levinson, for introducing us to this project and their support. 
 We also thank Jacob McCann for his contributions, and Emily King and Ewan Davies for their feedback.

\section{Notation and background}\label{sec:background}

\subsection{Partitions and tableaux}

A \textbf{partition} $\lambda=(\lambda_1\geq \ldots \geq \lambda_k)$ is a weakly decreasing sequence of nonnegative integers. The \textbf{Young diagram} of a partition $\lambda$ is a left-justified collection of boxes where $\lambda_i$ is the number of boxes in the $i$-th row. We denote the \textbf{size} of a partition as $|\lambda|=\sum \lambda_i$ and define the \textbf{length} of a partition $\ell(\lambda)=k$. 
We assume throughout that $\lambda$ fits inside a $k\times(n-k)$ rectangle, meaning $\ell(\lambda)\leq k$ and $\lambda_i\leq n-k$ for all $i$.

If $\rho$ and $\lambda$ are partitions such that $\rho_i\leq \lambda_i$ for all $i$, we say $\lambda$ \textbf{contains} $\rho$ and write $\rho\subseteq \lambda$. If $\rho\subseteq \lambda$, we define a \textbf{skew shape} $\lambda/\rho$ as the result of removing the boxes of $\rho$ from the diagram of $\lambda$. In the case that $\rho=\varnothing$, we refer to $\lambda$ as a \textbf{straight shape}.

Any way of placing positive integers in the boxes of a diagram is called a \textbf{filling}. A (skew) \textbf{ semistandard Young tableau} is a filling of a skew shape such that entries are weakly increasing along the rows and strongly increasing down the columns. A semistandard Young tableau is \textbf{standard} if each entry $\{1,\ldots ,k\}$ is used exactly once. All Young tableaux are assumed to be semistandard unless stated otherwise.

The \textbf{shape} of a tableau $T$ refers to the skew shape associated with $T$ and is denoted $\shape(T)$. If $S$ and $T$ are tableaux with shapes $\lambda/\rho$ and $\mu/\lambda$ respectively, then we say $T$ \textbf{extends} $S$. A \textbf{chain of Young tableaux} is a sequence $(\T_1,\ldots,\T_k)$ such that each $\T_i$ extends $\T_{i-1}$.

A \textbf{word} is a string $w=w_1w_2\ldots$ in the symbols $\{1,2,\ldots\}$. The \textbf{(row) reading word} of a tableau $T$ is the word $w(T)$ obtained by reading off the entries of $T$ from left to right starting from the bottom row. The \textbf{$i,i+1$-subword} of a word $w$ is the result of deleting any entry in $w$ not equal to $i$ or $i+1$. The \textbf{weight} of $T$ is $\weight(T)=(n_1,\ldots,n_k)$ where $n_i$ is the number of $i$ entries in $T$.

Let $X$ and $T$ be a pair of tableaux such that $T$ extends $X$, and let $n$ be the largest entry of $X$. We take their union,  $X\sqcup T$, to be the result of adding $n$ to each entry of $T$ and treating the pair as a single tableau.

The \textbf{standardization} of a word $w = w_1 \ldots w_n$ is the word $std(w)$ formed by replacing the letters by $1,\ldots,n$, from least to greatest, with ties broken in reading order. The standardization of a tableau is defined by standardizing its reading word.

\subsection{Jeu de taquin and equivalence relations}

Given a skew tableau $\lambda/\mu$, an \textbf{inner corner} is a box contained in the (deleted) diagram of $\mu$ such that the boxes below and to the right of it are not in $\mu$. An \textbf{outer corner} is a box of $\lambda$ such that the boxes below and to the right of it are not in $\lambda$.

If $T$ is a skew tableau of shape $\lambda/\mu$, an \textbf{inner jeu de taquin (JDT) slide} of $T$ into this empty box is the result of the following procedure: first, of the two boxes to the right and below of the empty box, slide the one with the smaller entry into the empty space. Should the entries in those boxes be tied, slide the box below. This process repeats until the empty box is now an outer corner of $\lambda$. An \textbf{outer JDT slide} refers to the inverse of an inner slide. 

A tableau $T$ is \textbf{rectified} if it is of straight shape. The \textbf{rectification} of $T$ is the tableau $\mathrm{rect}(T)$ formed by applying inner JDT slides on $T$ in any order until we obtain a straight shape tableau. By what is frequently called the ``Fundamental Theorem of JDT", $\mathrm{rect}(T)$ does not depend on the order in which we apply the slides.

\begin{definition}\label{def:evacuation}
    Let $T$ be a rectified semistandard tableau with largest entry $n$. The Sch\"{u}tzenberger involution, or \textbf{evacuation} of $T$, $e(T)$, can be constructed as follows. Delete the unique entry in the upper left corner of $T$; let $i$ be its value. Rectify the remaining entries of $T$, then fill the emptied box with $n+1-i$ as part of a new tableau $S$. Then repeat the above steps until each entry of $T$ has been replaced. 
\end{definition}
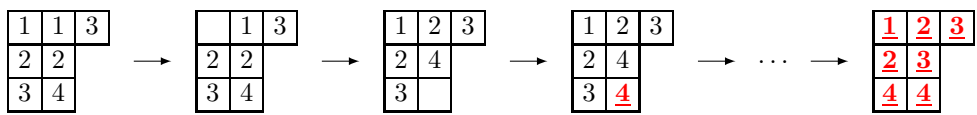
\[
    \begin{tikzpicture}
    \node {
    \begin{ytableau}
        1 & 1 & 3\\
        2 & 2\\
        3 & 4
     \end{ytableau}};
     \draw[-{latex}] (1,0) -- ++(.5,0);
     \node  at (2.5,0) {
    \begin{ytableau}
        \nonumber & 1 & 3\\
        2 & 2\\
         3 & 4
     \end{ytableau}};
     \draw[-{latex}] (3.5,0) -- ++(.5,0);
     \node  at (5,0) {
    \begin{ytableau}
        1 & 2 & 3\\
        2 & 4\\
         3 & \nonumber
     \end{ytableau}};
     \draw[-{latex}] (6,0) -- ++(.5,0);
     \node  at (7.5,0) {
    \begin{ytableau}
        1 & 2 & 3\\
        2 & 4\\
        3 & \color{red} \textbf{\underline{4}}
     \end{ytableau}};
     \draw[-{latex}] (8.5,0) -- ++(.5,0);
     \node at (9.5,0) {$\ldots$};
     \draw[-{latex}] (10,0) -- ++(.5,0);
     \node  at (11.5,0) {
    \begin{ytableau}
        \color{red}\textbf{\underline{1}} & \color{red}\textbf{\underline{2}} & \color{red}\textbf{\underline{3}}\\
        \color{red}\textbf{\underline{2}} & \color{red}\textbf{\underline{3}}\\
         \color{red}\textbf{\underline{4}} & \color{red} \textbf{\underline{4}}
     \end{ytableau}};
    \end{tikzpicture}
\]

Two skew semistandard Young tableaux $S$ and $T$ are \textbf{slide equivalent} ($S \slideEq T$) if there exists a sequence of JDT slides that takes one tableau to the other. Two semistandard Young tableaux are \textbf{dual equivalent} ($S=_{\DE}T$) if they have the same shape as one another after applying any sequence of JDT slides. In this language, the fundamental theorem of JDT is effectively equivalent to the statement that any two rectified tableaux of the same shape are dual equivalent. With respect to the RSK correspondence, $S$ and $T$ are slide equivalent if and only if they have the same insertion tableaux and dual equivalent if and only if they have the same shape and recording tableaux.

A \textbf{dual equivalence class} is an equivalence class of Young tableaux under dual equivalence. We may also define the \textbf{rectification shape} of a dual equivalence class $D$ to be the shape of $\mathrm{rect}(T)$ for any $T\in D$.

Given a tuple of partitions $\lambda_1,\ldots,\lambda_r$ such that $\sum|\lambda_i|=k(n-k)$, a \textbf{chain of dual equivalence classes} is a sequence $(D_1,\ldots,D_r)$ such that $D_i$ has rectification shape $\lambda_i$. We use $\DE(\lambda_1,\ldots,\lambda_r)$ to refer to the set of all such chains.

For a more in depth discussion on dual equivalence, see \cite{dual_equiv}.

\subsection{Littlewood-Richardson tableaux}\label{sec: LR-tab}

A word $w=w_1\ldots w_n$ is \textbf{ballot} if when read from the beginning to any letter, the sequence $w_1\ldots w_j$ contains at least as many $i$'s as $i+1$'s.

A word $w=w_1\ldots w_n$, is \textbf{reverse-ballot}  if when read backwards from the end to any letter, the sequence $w_n \ldots w_{n-j}$ contains at least as many $i$'s as $i+1$'s. A tableau whose reading word is reverse-ballot is called \textbf{Littlewood-Richardson} (LR); however, we may also say the tableau itself is reverse-ballot.

For a given straight shape $\lambda$, the \textbf{highest weight} tableau with shape $\lambda$ is the tableau with the $i$-th row filled entirely with $i$ entries. Each dual equivalence class with rectification shape $\lambda$ has a unique \textbf{highest weight representative} which rectifies to the unique highest weight tableau of straight shape $\lambda$.

\begin{proposition}
    A tableau is Littlewood-Richardson if and only if it is the highest weight representative of its dual equivalence class.
\end{proposition}

This is well known and provides a fixed representative for each dual equivalence class. In fact, we have a canonical bijection between chains of Littlewood-Richardson tableaux with weights (and rectification shapes) $\lambda_1\ldots \lambda_r$ and chains of dual equivalence classes with rectification shapes $\lambda_1\ldots \lambda_r$. This will allow for us to use Littlewood-Richardson tableaux for all of our computations.

\begin{center}
    $\LR(\lambda_1,\ldots,\lambda_r)\cong \DE(\lambda_1,\ldots,\lambda_r)$
\end{center}

For more on the general background of tableau combinatorics, we direct readers to Fulton \cite{Fulton}.

\subsection{Switching operations on tableaux and dual equivalence classes}\label{sec:switching-ops}

For this section, we assume $X$ and $T$ are semistandard Young tableaux such that $T$ extends $X$, and $D_1$ and $D_2$ are their respective dual equivalence classes. we will now define several operations which switch the relative positions $X$ and $T$.

\begin{definition}
    \textbf{Tableau switching}, also known as shuffling, a pair of tableau is achieved by performing successive inward JDT slides on $T$ in the order determined by the standardization order of $X$ from largest to smallest.
\end{definition}

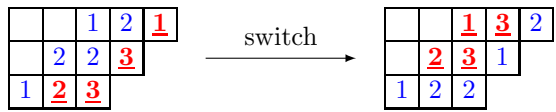
\[
    \begin{tikzpicture}
        \node at (0,0) {
        \begin{ytableau}
            \nonumber & & \color{blue}1 & \color{blue}2 & \color{red}\underline{\textbf{1}}\\
            \nonumber & \color{blue}2 & \color{blue}2 & \color{red}\underline{\textbf{3}}\\
            \color{blue}1 & \color{red}\underline{\textbf{2}} & \color{red}\underline{\textbf{3}}
        \end{ytableau}};
        \node at (2.5,.3) {$\shuffle$};
        \draw[-{latex}] (1.5,0) -- ++(2,0);
        \node at (5,0) {\begin{ytableau}
            \nonumber & & \color{red}\underline{\textbf{1}} & \color{red}\underline{\textbf{3}} & \color{blue}2\\
            \nonumber & \color{red}\underline{\textbf{2}} & \color{red}\underline{\textbf{3}} & \color{blue}1\\
            \color{blue}1 & \color{blue}2 & \color{blue}2
        \end{ytableau}};
    \end{tikzpicture}
\]
Switching descends to an operation on dual equivalence classes, which we may compute by switching any representatives of $D_1$ and $D_2$.

\begin{definition}
    A function on either Young tableaux or dual equivalence classes is \textbf{coplactic} if it commutes with any sequence of JDT slides. We note that switching is \textit{not} coplactic.
\end{definition}

We will introduce several coplactic operations related to switching in this section, and a more general class of coplactic operators in Section \ref{sec:coplactic-operators}.

\begin{definition}\label{def:coplactic_evacuation}
    Suppose $T$ has shape $\lambda/\mu$. Choose any rectified Young tableau $S$ with shape $\mu$. The \textbf{Sch\"{u}tzenberger involution}, or \textbf{coplactic extension of evacuation} of $T$, $e(T)$, is the result of conjugating evacuation by rectification as follows:

    \begin{itemize}
        \item \textbf{Rectify}. Compute $\shuffle(S,T)=(T',S')$.
        \item \textbf{Evacuate}. Compute $e(T')=T''$.
        \item \textbf{Unrectify}. Compute $\shuffle(T'',S')=(S'',T''')$.
    \end{itemize}

    We then define $\coplacticEvac(T) := T'''$, which does not depend on the choice of $S$.
    If $T$ is rectified, the definition reduces to Def. \ref{def:evacuation}. Hence, we use evacuation to also refer to its coplactic extension. By abuse of notation, we may write the coplactic extension of evacuation as
    \[
    e(T) = \mathrm{rect}^{-1} \circ e \circ \mathrm{rect}(T).
    \]
    We note that this involution is used in the theory of crystal graphs to show that crystals are self-dual as posets, see \cite[Theorem 5.11]{lenart2006combinatoricscrystalgraphsi}.
\end{definition}




\begin{remark}
    Let $w$ be a word. Then the evacuation of $w$ can be computed by evacuating the corresponding insertion tableau under the RSK correspondence (see \cite{Fulton}).
\end{remark}

We may construct multiple coplactic operations on Young tableaux by conjugating a function by rectification as described above; in each case, the output does not depend on the choice of auxiliary tableau $S$. To extend rectification to pairs of tableau, we define $\mathrm{rect}(X,T)$ to be the two parts of $\mathrm{rect}(X\sqcup T)$. We may similarly rectify a pair of dual equivalence classes by choosing tableau representatives for our computations, then taking the dual equivalence classes of the output.

\begin{definition}\label{def:evacuation-shuffling}
    The coplactic extension of switching a pair of rectified dual equivalence classes is called \textbf{evacuation shuffling} (esh), a coplactic involution defined as
    \begin{center}
        $\esh(D_1,D_2)=\mathrm{rect}^{-1}\circ \shuffle\circ \mathrm{rect}(D_1,D_2)$.
    \end{center}
\end{definition}

We may also extend esh to chains of dual equivalence classes, thereby reversing the order of some consecutive subset of dual equivalence classes. This is achieved by rectifying $D_i,\ldots,D_j$, then successively switching $D_i$ past $D_{i+1},\ldots,D_j$, then switching $D_{i+1}$ past the remaining $D_{i+2},\ldots,D_j$, and so on until the order has reversed. Afterwards, we unrectify the resulting sequence $D_j,\ldots,D_i$.

\begin{center}
    $\esh_{i,j}(D_1,\ldots,D_r)=(D_r,\ldots,D_{i-1},D_j,\ldots,D_i,D_{j+1},\ldots,D_r)$.
\end{center}


We now define several lifts of esh to functions on pairs of tableaux. Our first is tableau coswitching, which most closely resembles the definition of esh:

\begin{definition}\label{def:coswitch}
    \textbf{Tableau coswitching} (coswitch) is the coplactic extension of tableau switching from the case where the inner tableau is of straight shape, giving us
    \begin{center}
        $\coswitch(X,T)=\mathrm{rect}^{-1}\circ \shuffle\circ \mathrm{rect}(X,T)$.
    \end{center}
\end{definition}

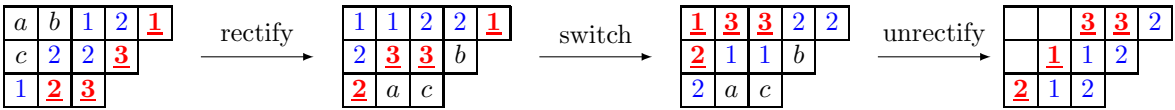
\[
    \begin{tikzpicture}
        \node at (0,0) {
        \begin{ytableau}
            a & b & \color{blue}1 & \color{blue}2 & \color{red}\textbf{\underline{1}}\\
            c & \color{blue}2 & \color{blue}2 & \color{red}\textbf{\underline{3}}\\
            \color{blue}1 & \color{red}\textbf{\underline{2}} & \color{red}\textbf{\underline{3}}
        \end{ytableau}};
        \node at (2.25,.3) {$\mathrm{rectify}$};
        \draw[-{latex}] (1.5,0) -- ++(1.5,0);
        \node at (4.5,0) {\begin{ytableau}
            \color{blue} 1 & \color{blue} 1 & \color{blue} 2 & \color{blue} 2 & \color{red} \textbf{\underline{1}} \\
            \color{blue} 2 & \color{red} \textbf{\underline{3}} & \color{red} \textbf{\underline{3}} & b\\
            \color{red} \textbf{\underline{2}} & a & c
        \end{ytableau}};
        \node at (6.75,.3) {$\shuffle$};
        \draw[-{latex}] (6,0) -- ++(1.5,0);
        \node at (9,0) {\begin{ytableau}
            \color{red} \textbf{\underline{1}} & \color{red} \textbf{\underline{3}} & \color{red} \textbf{\underline{3}} & \color{blue} 2 & \color{blue} 2 \\
            \color{red} \textbf{\underline{2}} & \color{blue} 1 & \color{blue} 1 & b\\
            \color{blue} 2 & a & c
        \end{ytableau}};
        \node at (11.25,.3) {$\mathrm{unrectify}$};
        \draw[-{latex}] (10.5,0) -- ++(1.5,0);
    \end{tikzpicture}\\
    \begin{tikzpicture}
        \node at (5,0) {\begin{ytableau}
            \nonumber & & \color{red} \textbf{\underline{3}} & \color{red} \textbf{\underline{3}} & \color{blue} 2\\
            \nonumber & \color{red} \textbf{\underline{1}} & \color{blue} 1 & \color{blue} 2\\
            \color{red} \textbf{\underline{2}} & \color{blue} 1 & \color{blue} 2
        \end{ytableau}};
    \end{tikzpicture}
\]

Coswitching can also be extended to chains of tableaux in the same way esh does; specifically we rectify a chain of tableaux $T_i,\ldots T_j$, then switch $T_i$ past the remaining $T_{i+1},\ldots T_j$, and repeat until the order is reversed, then unrectify the resulting chain of tableaux.

\begin{definition}
    Our next involution is \textbf{evacuation} from Definition \ref{def:coplactic_evacuation}. By abuse of notation, we take $\coplacticEvac(X,T)$ to be the two parts of $\coplacticEvac(X\sqcup T)$. (It is a fact that the result is naturally of the form $T' \sqcup X'$, giving a pair of tableaux with the desired rectification shapes.)
\end{definition}

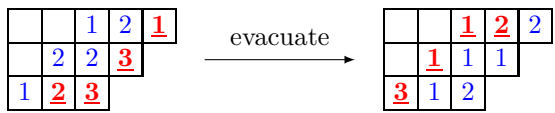
\[
    \begin{tikzpicture}
        \node at (0,0) {
        \begin{ytableau}
            \nonumber & & \color{blue}1 & \color{blue}2 & \color{red}\underline{\textbf{1}}\\
            \nonumber & \color{blue}2 & \color{blue}2 & \color{red}\underline{\textbf{3}}\\
            \color{blue}1 & \color{red}\underline{\textbf{2}} & \color{red}\underline{\textbf{3}}
        \end{ytableau}};
        \node at (2.5,.3) {$\mathrm{evacuate}$};
        \draw[-{latex}] (1.5,0) -- ++(2,0);
        \node at (5,0) {\begin{ytableau}
            \nonumber & & \color{red}\underline{\textbf{1}} & \color{red}\underline{\textbf{2}} & \color{blue}2\\
            \nonumber & \color{red}\underline{\textbf{1}} & \color{blue}1 & \color{blue}1\\
            \color{red}\underline{\textbf{3}} & \color{blue}1 & \color{blue}2
        \end{ytableau}};
    \end{tikzpicture}
\]

While $\coplacticEvac{}$ evacuates the slide classes of $X$ and $T$, switching and $\coswitch$ing both preserve the slide equivalences of the tableaux. Our final lift of esh will instead evacuate the slide class of the inner tableau and preserve that of the outer.

\begin{definition}
    We define \textbf{\algorithm{}} ($\alg$) to be 
    \[\alg(X,T)=\coswitch(\coplacticEvac(X),T).\]
\end{definition}

\begin{center}
    \begin{tikzpicture}
        \node at (0,0) {
        \begin{ytableau}
            \nonumber & & \color{blue}1 & \color{blue}2 & \color{red}\underline{\textbf{1}}\\
            \nonumber & \color{blue}2 & \color{blue}2 & \color{red}\underline{\textbf{3}}\\
            \color{blue}1 & \color{red}\underline{\textbf{2}} & \color{red}\underline{\textbf{3}}
        \end{ytableau}};
        \node at (2.5,.3) {$\alg$};
        \draw[-{latex}] (1.5,0) -- ++(2,0);
        \node at (5,0) {\begin{ytableau}
            \nonumber & & \color{red}\underline{\textbf{3}} & \color{red}\underline{\textbf{3}}& \color{blue}2\\
            \nonumber & \color{red}\underline{\textbf{1}}& \color{blue}1 & \color{blue}1\\
            \color{red}\underline{\textbf{2}} & \color{blue}1 & \color{blue}2
        \end{ytableau}};
    \end{tikzpicture}
\end{center}

Every operation defined so far is an involution, except for \algorithm{} which has order 4. We also observe that $\alg{}$ has the same effect on $X$ as $\coplacticEvac{}$ and the same effect on $T$ as $\coswitch{}$, see the example above where the blue tableau above matches that of $\coplacticEvac{}$ and the red underlined tableau matches that of $\coswitch$.

\subsection{Connection to smooth covers of \texorpdfstring{$\MoNBar$}{M0,n-bar}}\label{sec:geometry}

We briefly discuss where our tableau bijections arise in geometry of certain covering spaces of the Deligne--Mumford moduli space $\MoNBar(\mathbb{R})$ of real $r$-marked stable curves of genus $0$. See \cite{1box_b, 1box_a, Levinson_2017, speyer} for details; the goal of this section is simply to highlight where each bijection arises in such a covering.

Fix integers $k \leq n$ and a tuple $\lambda_1, \ldots, \lambda_r$ of partitions such that $\sum |\lambda_i| = k(n-k)$. Speyer \cite{speyer} introduced a smooth covering space
\[\mathcal{S} = \mathcal{S}(\lambda_1, \ldots, \lambda_r)(\mathbb{R}) \to \MoNBar(\mathbb{R})\]
whose monodromy is described by (variations of) evacuation shuffling on dual equivalence classes and tableaux. Specifically, the interior $\MoNBar(\mathbb{R})$ is the disjoint union of $(r-1)!/2$ contractible cells, corresponding to arrangements of $1, \ldots, r$ on a necklace (i.e. on a circle, up to rotation and reflection). Above each such cell, the corresponding sheets of $\mathcal{S}$ are indexed by chains of dual equivalence classes. For the ordering $\sigma(1), \ldots, \sigma(r)$, the sheets of $\mathcal{S}$ are in bijection with $\DE(\lambda_{\sigma(1)}, \ldots, \lambda_{\sigma(r)})$. An important special case is when several or all of the partitions are single boxes; a chain of dual equivalence classes of types $\ybox_1, \ybox_2, \ldots, \ybox_t$ is just the data of a (standard) Young tableau of size $t$.

The \textbf{cactus group} $J_{r-1}$ is generated by elements $s_{ij}$ satisfying certain cactus relations (see e.g. \cite{henriques2005crystalscoboundarycategories}). Each $s_{ij}$ corresponds to a path from one cell of $\MoNBar$ to an adjacent one by reversing the order of the $i$-th through $j$-th points around the circle. This allows us to generate the fundamental groupoid of $\MoNBar$ via compositions of the $s_{ij}$'s.
Concretely, reversing the ordering from 
\[1, \ldots, i-1, \underbrace{i, \ldots, j}, j+1, \ldots, r \hspace{3em} \mathrm{to} \hspace{3em} 1, \ldots, i-1, \underbrace{j, \ldots, i}, j+1, \ldots, r\]
corresponds (by monodromy on the sheets of $\mathcal{S}$) to the bijection $\esh_{i,j}$ on chains of dual equivalence classes
\[
\mathrm{esh}_{i,j}:\DE(\lambda_1, \ldots, \lambda_r) \xrightarrow{\sim} 
\DE(\lambda_1, \ldots, \lambda_{i-1}, \lambda_j, \ldots, \lambda_i, \lambda_{j+1} \ldots, \lambda_r).
\]
Because of the canonical bijection between dual equivalence classes and Littlewood-Richardson tableaux, we may use Littlewood Richardson tableaux in place of dual equivalence classes and coswitching in place of esh (as coswitching is the only coplactic operation which preserves ballotness).
\[
\mathrm{\coswitch}_{i,j}:\LR(T_1, \ldots, T_r) \xrightarrow{\sim} 
\LR(T_1, \ldots, T_{i-1}, T_j, \ldots, T_i, T_{j+1} \ldots, T_r).
\]

In this setting, pesh and evacuation arise as special cases of coswitching, where we represent a chain of single boxes as a skew standard Young tableau.

\begin{center}
{\renewcommand{\arraystretch}{1.8}
\begin{tabular}{|c|c|c|}
    \hline
    Wall-crossing & Tableau data & Bijection\\
    \hline
    $(\ldots, \lambda, \mu, \ldots) \leadsto (\ldots, \mu, \lambda, \ldots)$ & $X,T\in \LR$ & $\coswitch(X,T)$\\
    \hline
    $(\ldots,\underbrace{\ybox_1, \ldots, \ybox_t},\ldots) \leadsto (\ldots,\underbrace{\ybox_t, \ldots, \ybox_1},\ldots)$ & $X\in\SSYT$ & $e(X)$\\
    \hline
    $(\ldots,\underbrace{\ybox_1, \ldots, \ybox_t,\lambda},\ldots) \leadsto (\ldots,\underbrace{\lambda, \ybox_t, \ldots, \ybox_1},\ldots)$ & $X\in\SSYT$, $T\in\LR$ & $\alg(X,T)$\\
    \hline
\end{tabular}}
\end{center}

\begin{example}
    Consider the covering space $\mathcal{S} = \mathcal{S}(\alpha, \beta, \gamma, \delta) \to \overline{M}_{0, 4}(\mathbb{R})$, where $|\alpha| + |\beta| + |\gamma| + |\delta| = k(n-k)$ for $k=3$ and $n=8$, of degree equal to the corresponding Littlewood-Richardson coefficient.

    Let $C_{abcd} \subset M_{0, 4}(\mathbb{R})$ denote the cell on which the four marked points are in the circular order $a,b,c,d$. The path $s_{23}$ thus connects the cell $C_{1234}$ to $C_{1324}$. Over these two cells, the sheets of $\mathcal{S}$ are indexed, respectively, by $\LR(\alpha, \beta, \gamma, \delta)$ and by $\LR(\alpha, \gamma, \beta, \delta)$. The path-lifting of $s_{23}$ corresponds to the bijection
    \begin{align*}
    \LR(\alpha, \beta, \gamma, \delta) &\to \LR(\alpha, \gamma, \beta, \delta) \\
    (X_1, X_2, X_3, X_4) &\mapsto (X_1, \coswitch(X_2, X_3), X_4).
    \end{align*}

    \begin{figure}[h]
    \centering
    \begin{tikzpicture}
        \node at (-1.1,-.35) {\begin{ytableau}
            \none & \none & \color{blue}1 & \color{blue}1 & \color{red}\textbf{\underline{1}}\\
            \none & \color{blue}1 & \color{blue}2 & \color{red}\textbf{\underline{2}}\\
            \color{blue}1 & \color{red}\textbf{\underline{1}} & \color{red}\textbf{\underline{1}}
        \end{ytableau}};
        \draw (-2.215,-.58) -- (-2.215,.32) -- (-1.3,.32);
        \node at (-1.95,.05) {$\alpha$};
        \draw (-.9,-1.02) -- (0.02,-1.015) -- (0.015,-.1);
        \node at (-.25,-.75) {$\delta$};
        
        \node at (8.5, -1.8) {\begin{ytableau}
            \none & \none & \color{red} \textbf{\underline{1}} & \color{red} \textbf{\underline{1}} & \color{blue} 1\\
            \none & \color{red} \textbf{\underline{1}} & \color{blue} 1 & \color{blue} 1\\
            \color{red} \textbf{\underline{2}} & \color{blue} 1 & \color{blue} 2
        \end{ytableau}};
        \draw (7.385,-2.05) -- (7.385,-1.13) -- (8.28,-1.13);
        \node at (7.65,-1.4) {$\alpha$};
        \draw (8.7,-2.47) -- (9.62,-2.465) -- (9.615,-1.55);
        \node at (9.35,-2.2) {$\delta$};

        \node at (-4,-0.875) {$S(\mathbb{R})$};
        \draw [-{Latex[length=.1in]}] (-4,-1.2) -- (-4,-3.6);
        \node at (-4,-4) {$\MoNBar$};
        \draw (1.15,-.35) ellipse (1cm and .5cm);
        \draw (1.15,-1.8) ellipse (1cm and .5cm);
        \filldraw (1.15,-.35) circle (.05);
        \filldraw (6.15,-.35) circle (.05);
        \filldraw (1.15,-1.8) circle (.05);
        \filldraw (6.15,-1.8) circle (.05);
        \draw (6.15,-.35) ellipse (1cm and .5cm);
        \draw (6.15,-1.8) ellipse (1cm and .5cm);
        \draw (1.15,-.35) to [bend left=15] (3.65,-1.075);
        \draw [-{Latex}] (3.65,-1.075) to [bend right=15] (6.07,-1.8);
        \draw (1.15,-1.8) to [bend right=15] (3.65,-1.075);
        \draw [-{Latex}] (3.65,-1.075) to [bend left=15] (6.07,-.35);
        \draw (1.15,-4) ellipse (1cm and .5cm);
        \draw (6.15,-4) ellipse (1cm and .5cm);
        \filldraw (1.15,-4) circle (.05);
        \filldraw (6.15,-4) circle (.05);
        \draw [-{Latex}] (1.15,-4) -- (6.07,-4);
    \end{tikzpicture}
    \caption{An example of $\esh$ as a lift of the path $s_{2,3}$ in $\overline{M}_{0,4}$.}
    \label{fig:path_lift}
\end{figure}
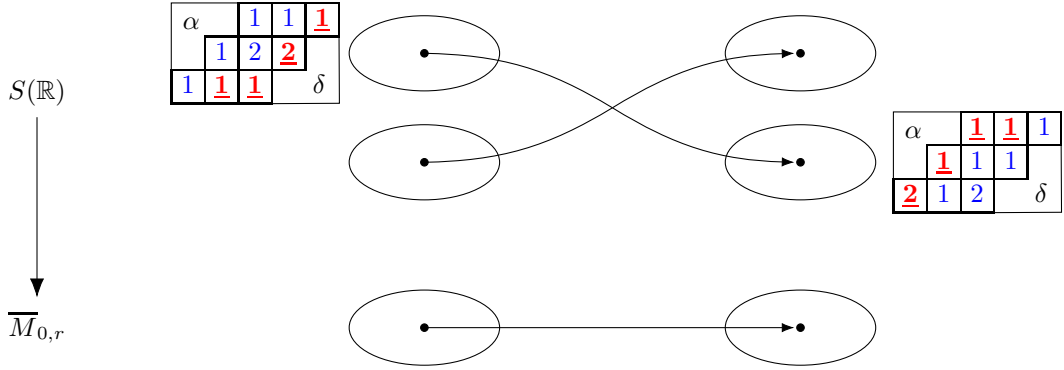
\end{example}

\begin{example}
Consider a similar example where instead of the partition $\beta$, we have a chain of single box partitions $\ybox_1,\ldots,\ybox_5$. This means we now have the covering space $\mathcal{S} = \mathcal{S}(\alpha, \ybox_1,\ldots,\ybox_5, \gamma, \delta) \to \overline{M}_{0, 8}(\mathbb{R})$

The path-lifting of $s_{27}$ corresponds to the bijection
    \begin{align*}
    \LR(\alpha, \ybox_1,\ldots,\ybox_5, \gamma, \delta) &\to \LR(\alpha, \gamma, \ybox_5,\ldots,\ybox_1, \delta) \\
    (X_1, X_2, \ldots, X_8) &\mapsto (X_1, \coswitch(X_2, \ldots,X_7), X_8).
    \end{align*}

However, compared to the previous example, we are coswitching a chain of single boxes past a single partition $\gamma$. The data of a chain of single box partitions is equal to the data of a standard Young tableau, so we may instead compute this path lift using $\alg$:
\begin{align*}
    \LR(\alpha, \ybox_1,\ldots,\ybox_5, \gamma, \delta) &\to \LR(\alpha, \gamma, \ybox_5,\ldots,\ybox_1, \delta) \\
    (X_1, X_2, X_3, X_4) &\mapsto (X_1, \alg(X_2, X_3), X_4).
    \end{align*}

    \begin{figure}[h]
    \centering
    \begin{tikzpicture}
        \node at (-1.1,-.35) {\begin{ytableau}
            \none & \none & \color{blue}3 & \color{blue}4 & \color{red}\textbf{\underline{1}}\\
            \none & \color{blue}2 & \color{blue}5 & \color{red}\textbf{\underline{2}}\\
            \color{blue}1 & \color{red}\textbf{\underline{1}} & \color{red}\textbf{\underline{1}}
        \end{ytableau}};
        \draw (-2.215,-.58) -- (-2.215,.32) -- (-1.3,.32);
        \node at (-1.95,.05) {$\alpha$};
        \draw (-.9,-1.02) -- (0.02,-1.015) -- (0.015,-.1);
        \node at (-.25,-.75) {$\delta$};
        
        \node at (8.5, -1.8) {\begin{ytableau}
            \none & \none & \color{red} \textbf{\underline{1}} & \color{red} \textbf{\underline{1}} & \color{blue} 5\\
            \none & \color{red} \textbf{\underline{1}} & \color{blue} 1 & \color{blue} 4\\
            \color{red} \textbf{\underline{2}} & \color{blue} 2 & \color{blue} 3
        \end{ytableau}};
        \draw (7.385,-2.05) -- (7.385,-1.13) -- (8.28,-1.13);
        \node at (7.65,-1.4) {$\alpha$};
        \draw (8.7,-2.47) -- (9.62,-2.465) -- (9.615,-1.55);
        \node at (9.35,-2.2) {$\delta$};

        \node at (-4,-0.875) {$S(\mathbb{R})$};
        \draw [-{Latex[length=.1in]}] (-4,-1.2) -- (-4,-3.6);
        \node at (-4,-4) {$\MoNBar$};
        \draw (1.15,-.35) ellipse (1cm and .5cm);
        \draw (1.15,-1.8) ellipse (1cm and .5cm);
        \filldraw (1.15,-.35) circle (.05);
        \filldraw (6.15,-.35) circle (.05);
        \filldraw (1.15,-1.8) circle (.05);
        \filldraw (6.15,-1.8) circle (.05);
        \draw (6.15,-.35) ellipse (1cm and .5cm);
        \draw (6.15,-1.8) ellipse (1cm and .5cm);
        \draw (1.15,-.35) to [bend left=15] (3.65,-1.075);
        \draw [-{Latex}] (3.65,-1.075) to [bend right=15] (6.07,-1.8);
        \draw (1.15,-1.8) to [bend right=15] (3.65,-1.075);
        \draw [-{Latex}] (3.65,-1.075) to [bend left=15] (6.07,-.35);
        \draw (1.15,-4) ellipse (1cm and .5cm);
        \draw (6.15,-4) ellipse (1cm and .5cm);
        \filldraw (1.15,-4) circle (.05);
        \filldraw (6.15,-4) circle (.05);
        \draw [-{Latex}] (1.15,-4) -- (6.07,-4);
    \end{tikzpicture}
    \caption{An example of $\alg$ as a lift of the path $s_{2,7}$ in $\overline{M}_{0,8}$. Note that in the output, the blue entries $1,\ldots, 5$ of the standard Young tableau, correspond to the boxes $\protect\ybox_5,\ldots,\protect\ybox_1$ (in that order).}
    \label{fig:path_lift_pesh}
\end{figure}
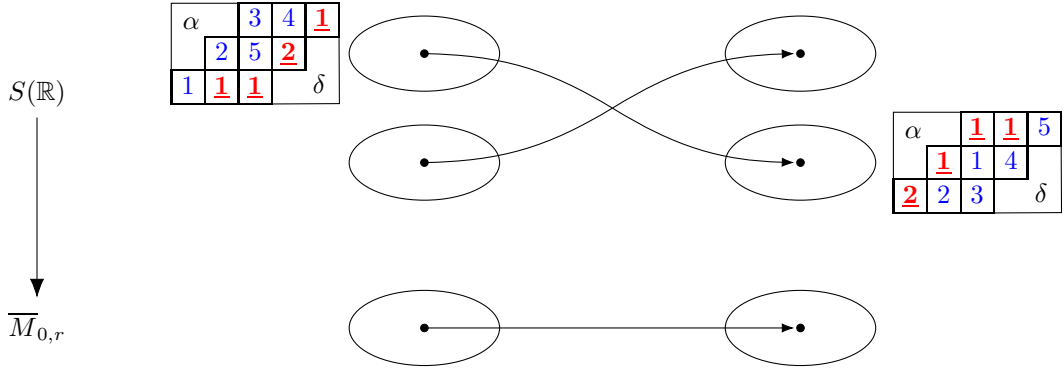
\end{example}



\subsection{Coplactic operators}\label{sec:coplactic-operators}

\textbf{Crystals} in type A are a set of weight raising and weight lowering coplactic operators, $E_i$ and $F_i$ respectively, that act on both semistandard Young tableaux and their reading words. 

\begin{definition}
Let $w$ be a word. Then the crystal operators $E_i(w)$ and $F_i(w)$ can be computed as follows. Consider the $i,i+1$-subword of $w$ with each $i$ and $i+1$ replaced with a ) and ( respectively. Starting from the end of the subword, pair off each ( with the closest unpaired ) to the right if it exists. If there is an unmatched (, then $E_i(w)$ changes the first unmatched ( to a ), otherwise $E_i(w)=\emptyset$. Similarly, $F_i(w)$ will instead change the last unmatched ) to a ( if it exists, otherwise $F_i(w)=\emptyset$.
\end{definition}

\begin{example}
    Consider the word $w=2221132122131$. Replacing each entry in the $1,2$-subword with parentheses gives

    \begin{center}
        \color{red}\textbf{(}\color{black}(())()(()).
    \end{center}
    The single bold red ( is the only unmatched symbol, so $F_1(w)$ is undefined while $E_1(w)=1221132122131$.
    
\end{example}

Let $T$ be a semistandard tableau with $n$ as the largest entry, and let $D$ be its dual equivalence class. Crystal operators preserve dual equivalence, and every pair of elements in a dual equivalence class is connected via crystal operators. In particular, the highest weight representative of $D$ can be computed by applying every possible $E_i$ crystal to $T$ for all $i\geq 1$ until no such weight raising operator is defined. Similarly, the \textbf{lowest weight} representative is the result of applying as many $F_i$ crystal operators as possible for $1\leq i<n$. This gives us a tableau with weight $\{0,\ldots,0,w_i,\ldots,w_{n}\}$, so we then subtract $i-1$ from each entry to obtain a tableau with weight $\{w_i,\ldots,w_{n}\}$.

\begin{proposition}
Let $T$ be a Littlewood-Richardson tableau. Then $F_i$ is defined if and only if the $(i,i+1)$ subword contains strictly more $i$'s than $(i+1)$'s.
\end{proposition}

\begin{proof}
    Crystal operators are coplactic, meaning $F_i(T)$ is defined if and only if $F_i(\mathrm{rect}(T))$ is defined. Observe that the $(i,i+1)$-reading word of $\mathrm{rect}(T)$ is of the form $(i+1)\ldots (i+1)i\ldots i$. Hence, there is an unmatched $i$ if if and only if there are more $i$'s than $(i+1)$'s.
\end{proof}


\section{Main results}\label{sec:main results}

\subsection{Local evacuation shuffling}\label{sec:algorithm-definition-a}

We now define two local algorithms to compute the bijections of Section \ref{sec:switching-ops}, meaning we act internally on the skew tableaux without rectifying or unrectifying. The first algorithm uses JDT-like local moves, while the second uses a composition of crystal operators.

For a word $w = w_1 \ldots w_n$, we say the suffix $w_{n-j} \cdots w_n$ is \textbf{tied for $(i,i+1)$} if it contains exactly as many $i$'s as $i+1$'s.

\begin{definition}\label{def:type_a_hop}
    (Type A hopping algorithm) Let $X$ be a standard Young tableau such that $|X|=n$, and let $T$ be a Littlewood-Richardson semistandard Young tableau. We refer to the entries of $X$ as $x_1, \ldots, x_n$ and the entries of $T$ by their numerical values. Let $d$ be the largest value of $T$.

    Then the hopping algorithm in Type A is defined as follows:

    Set $i=n$ and $j=1$. 

    \begin{itemize}
        \item \textbf{Phase 1a:} If $x_i$ precedes all of the $j$s of $T$ in reading order, replace $x_i$ by $j$, record $a_{n-i+1}=j$, and proceed to Phase 1b. (We say $x_i$ \emph{transitions} out of Phase $1$ and is absorbed into $T$ as a $j$.) If $x_i$ does not precede all of the $j$s in reading order, switch $x_i$ with the $j$ immediately preceding it in reading order. Increment $j$ and repeat Phase 1a.

        \item \textbf{Phase 1b:} If $i=1$, proceed to Phase 2. If $i\neq 1$, decrement $i$, reset $j=1$, and repeat Phase 1a.
        
        \item \textbf{Phase 2:} Let $j = a_{n-i+1}$. Replace the first $j$ in reading order with $x_{n-i+1}$. If its suffix is tied for $(j,j+1)$, increment $j$. If not, switch $x_{n-i+1}$ with the nearest $j$ after $x_{n-i+1}$ in reading order whose suffix is tied for $(j,j+1)$, and increment $j$. In either case, if the new value for $j=d+1$, increment $i$ and repeat Phase 2 with that $i$ value. If not, repeat Phase 2 with the original $i$ value.
    \end{itemize}
\end{definition}

We demonstrate this algorithm in the following example. 

\begin{example}\label{ex:type-a-hop}
    \[
    \begin{tikzpicture}
        \node at (0,0) {
        \begin{ytableau}
            \nonumber & x_1 & x_3 & 1 \\
            x_2 & 1 & 1\\
            1 & 2 & 2
        \end{ytableau}};
        \node at (2,.3) {Phase 1};
        \node at (2,-.3) {$a_1=3$};
        \draw[-{latex}] (1.25,0) -- ++(1.5,0);
        \node at (4,0) {\begin{ytableau}
            \nonumber & x_1 & 1 & 1 \\
            x_2 & 1 & 2\\
            1 & 2 & 3
        \end{ytableau}};
        \node at (5.75,.3) {Phase 1};
        \node at (5.75,-.3) {$a_2=2$};
        \draw[-{latex}] (5,0) -- ++(1.5,0);
        \node at (7.5,0) 
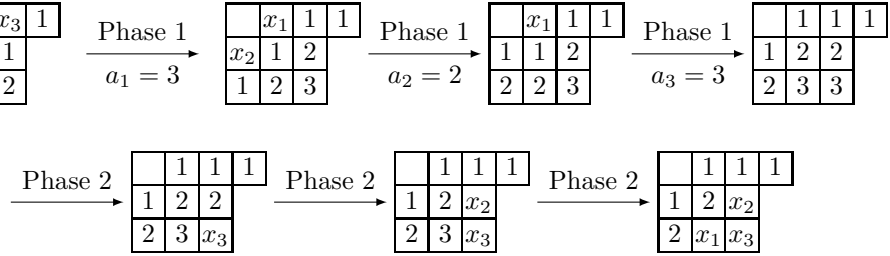
{\begin{ytableau}
            \nonumber & x_1 & 1 & 1\\
            1 & 1 & 2\\
            2 & 2 & 3
        \end{ytableau}};
        \node at (9.25,.3) {Phase 1};
        \node at (9.25,-.3) {$a_3=3$};
        \draw[-{latex}] (8.5,0) -- ++(1.5,0);
        \node at (11,0) {\begin{ytableau}
            \nonumber & 1 & 1 & 1\\
            1 & 2 & 2\\
            2 & 3 & 3
        \end{ytableau}};
    \end{tikzpicture}
\]
\[
\begin{tikzpicture}
        \draw[-{latex}] (4.5,0) -- ++(1.5,0);
        \node at (7,0) {\begin{ytableau}
            \nonumber & 1 & 1 & 1\\
            1 & 2 & 2\\
            2 & 3 & x_3
        \end{ytableau}};
        \node at (5.25,.3) {Phase 2};
        \draw[-{latex}] (8,0) -- ++(1.5,0);
        \node at (10.5,0) {\begin{ytableau}
            \nonumber & 1 & 1 & 1\\
            1 & 2 & x_2\\
            2 & 3 & x_3
        \end{ytableau}};
        \node at (8.75,.3) {Phase 2};
        \draw[-{latex}] (11.5,0) -- ++(1.5,0);
        \node at (14,0) {\begin{ytableau}
            \nonumber & 1 & 1 & 1\\
            1 & 2 & x_2\\
            2 & x_1 & x_3
        \end{ytableau}};
        \node at (12.25,.3) {Phase 2};
    \end{tikzpicture}
\]
\end{example}

\begin{remark}
    We may extend these algorithms to the case where $X$ is any semistandard tableau by acting on its entries in standardization order. Additionally, the hopping algorithm may be computed using only the combined reading word of $X\sqcup T$.
\end{remark}

\begin{lemma}
    The hopping algorithm is coplactic.
\end{lemma}

\begin{proof}
    The hopping algorithm is comprised entirely of standardization-preserving relabeling steps and Pieri switches, which were shown to be coplactic in \cite{1box_a}. Note that the condition of when to transition to Phase 2 (i.e. that no Pieri switch can be performed) is itself a coplactic condition \cite[Proposition 4.13]{1box_a}.
\end{proof}

We note that Phase 2 implicitly requires the existence of a suitable $j$, after $x_{n-i+1}$ in reading order, with which to switch $x_{n-i+1}$. To show that the algorithm is well-defined, it therefore suffices to show that the intermediate tableau $T$, excluding the `moving' entry, remains Littlewood-Richardson throughout the algorithm. (In particular, the suffix from the last $j$ in reading order is tied for $(j, j+1)$ as it contains no $j$'s nor $j+1$'s.) We prove this in the lemma below. A \textbf{step} refers to a single switch performed during the algorithm.

\begin{lemma}{\label{alwaysballot}}
    Let $X$ be a $\SSYT$ and $T$ be a $\mathrm{LR}$ tableau extending $X$. Let $S=X'\sqcup T'$ be a tableau that appears after some step of the computation of $\hop(X,T)$, where $X'$ consists of the entries of $X$ that have not reached Phase 1b, or that have completed Phase 2; and $T'$ consists of all other entries. Then:

\begin{enumerate}
    \item $T'$ is Littlewood-Richardson.
    \item $T'$ is semistandard. 
    \item Let $x_i$ be the entry that most recently moved in the algorithm and let $j$ be the index of the move. Then $x_i$ is an inner corner of the sub-tableau $T'_{>j}$ containing only those entries greater than $j$.
\end{enumerate}
\end{lemma}

\begin{proof}
    As the entries of $X$ perform Phase 1 or Phase 2 hops, all three statements essentially follow from \cite[Theorem 3.9]{1box_a}. However, it remains to be shown that the Transition Phase, which we alter to absorb an entry of $X$ into $T$, produces a tableau $S$ satisfying the first two claims. We already know $S$ is semistandard by the third statement, meaning we must only prove $S$ is Littlewood-Richardson:

    It suffices to show that when some $x_j$ transitions out of Phase 1 and is absorbed into $T$ as an $i$ entry, the resulting tableau omitting $x_1,\ldots ,x_{i-1}$ is Littlewood-Richardson. Because ballotness is preserved by JDT and the algorithm is coplactic, we may also simplify to the rectified case.
    
    Let $n\geq j$. Suppose $X$ is a single box $x_n$. Notice that because $X$ is rectified, the marked entry $x_n$ ends on row $j$ when it transitions out of Phase 1 at index $j$ \cite[Lemma 4.15]{1box_a}. Using \cite[Theorem 3.9]{1box_a}, we know that $T$ omitting $x_i$ is Littlewood-Richardson and semistandard. So, absorbing $x_i$ as a $j$ entry on row $i$ will preserve ballotness for $T$.

    We repeat this argument for $x_{n-1},\ldots x_i$, giving us a reverse-ballot tableau omitting $x_1,\ldots x_{i-1}$ and completing the proof.
\end{proof}

The third statement of Lemma \ref{alwaysballot} also shows that the output of $\hop(X,T)$ is, indeed, of the form $(T',X')$.

It turns out to be most natural to prove that the composition of the hopping algorithm with evacuation computes coswitching. Simple variations then compute $\alg$ and coplactic evacuation entirely in terms of switches (see Corollary \ref{cor:pesh-recovery}). To do so, we use the following to compare the outputs of $\hop(X,T)$ and $\coswitch(X,T)$:

\begin{definition}
    The \textbf{row data} of $\coswitch(X,T)$ is $(r(x_1),\ldots,r(x_n))$ where $r(x_i)$ is the row in which $x_i$ falls after $\coswitch(X,T)$. The \textbf{transition data} of $\hop(X,T)$ is $h(X,T)=(h(y_n),\ldots,h(y_1))$ where $h(y_i)$ is the number with which $y_i$ is replaced during Phase 1a of $\hop(X,T)$.
\end{definition}

\begin{lemma}{\label{permutations}}
    For a straight shape $X\sqcup T$, row data, $r(X,T)$, and transition data, $h(X,T)$, are permuations of each other.
\end{lemma}

\begin{proof}
    Let the output of $\coswitch(X,T)$ be $(T',X')$, and let the output of $\hop(X,T)$ be $(T'',X'')$.
    
    By Lemma \ref{alwaysballot}, $\coswitch(X,T)$ and $\hop(X,T)$ preserve both the ballotness and weight of $T$. Therefore, since $T'$ and $T''$ are both Littlewood-Richardson of highest weight, $T'=T''$. Because $\coswitch(X,T)$ and $\hop(X,T)$ preserve the overall tableau shape, we also know that $\shape(T',X')=\shape(T'',X'')$. Therefore, since $T'=T''$, we get that $\shape(X')=\shape(X'')$. Since we are working in the straight shape case, $h(x_i)$ is simply the row on which $x_i$ ends up since Phase 2 of $\hop(X,T)$ will just slide $x_i$ to the end of the $h(x_i)$'th row.
    
    Finally, we observe that $\weight(r(X,T))$ and $\weight(h(X,T))$ record the number of entries of $X'$ and $X''$ respectively on each row. Since $\shape(X')=\shape(X'')$, this means that that $\weight(r(X,T))=\weight(h(X,T))$. Thus, $r(X,T)$ and $h(X,T)$ must be permutations of each other.
\end{proof}

\begin{theorem} \label{thm: hop=esh}
    Let $Y$ be the tableau resulting from the evacuation $X$ with entries labeled $y_n,\ldots,y_1$ in order of evacuation. 
    The output of $\coswitch(X,T)$ and the output of $\hop(Y,T)$ agree.
\end{theorem}

\begin{proof}
    We will work in the case where $X$ is straight shape. We proceed via induction on $|X|$. 
    
    Our base case is $|X|=1$, so $X=x_1$. Then we have row data $(r(x_1))$ and transition data $(h(y_1))$, so by Lemma \ref{permutations}, the values $\coswitch(x_1,T)$ and $\hop(x_1,T)$ must agree.

    Assume the two agree for $|X|=n-1$ and consider the case where $|X|=n$. We perform one step of $\coswitch(X,T)$ which will move $x_n$ past $T$, leaving us with a chain of tableaux $(X',T',x_n)$ where $|X'|=n-1.$

    If we instead were to evacuate $X$, leaving us with the chain $(Y,T)$, $y_1$ must be the entry in the innermost corner. We remove $y_1$ and rectify into the empty space, which leaves us with $(Y',T'')$, where $|Y'|=n-1$. Note that $T'$ is the same as $T''$ since the content of $T'$ and $T''$ have not changed and the shape of $X'$ matches the shape of $Y'$. 
    
    Since evacuation is an involution and $e(X)=Y$, we know that $e(Y)=X$. Thus, this step of removing $y_1$ from $Y$ and rectifying to get $Y'$, which matches with the first step of computing $e(Y)$, must evacuate the same outer corner as the outer corner as the location of $x_n$ in $X$. Evacuating $Y'$ from this point is just the continuation of evacuating $Y$, so $e(Y')$ must match all of $X$ except for $x_n$, but this is simply $X'$. Thus, $e(Y')=X'$, and since evacuation is an involution, $e(X')=Y'$.

    By our inductive hypothesis then, we know that $\coswitch(X',T')$ agrees with $\hop(X',T'')$, and we have row data $(r(x_1),\ldots,r(x_{n-1}))$ and transition data $(h(y_n),\ldots,h(y_2))$ which, by Lemma \ref{permutations}, are permutations of each other. But we began by coswitching $x_n$ past $T$, which gives us row data $(r(x_1),\ldots,r(x_n))$. Because our row data and transition data must be permutations of each other, we conclude that $r(x_n)=h(y_1)$. Since the final step of the hopping algorithm replaces $y_1$ by $x_n$, coswitching and the hopping algorithm agree for $|X|=n.$

    Notice that all the steps of our algorithm involve Pieri switches and replacement of $x_i$ entries by numbers. We know the Pieri switches to be coplactic from Gillespie-Levinson \cite{1box_a}, and the replacements preserve both the semistandard and ballot requirements of $T$ at any point in our algorithm. This means that these replacements are also coplactic, so each step of our algorithm commutes with JDT slides. Thus, the output of $\coswitch(X,T)$ and the output of $\hop(X,T)$ still agree when $X\sqcup T$ is skew shape.
\end{proof}

We demonstrate this result in the two examples by first computing $\hop(Y,T)$ on the same pair of tableaux as in Example 3.2 and then computing $\coswitch(X,T)$ to observe that their outputs match.

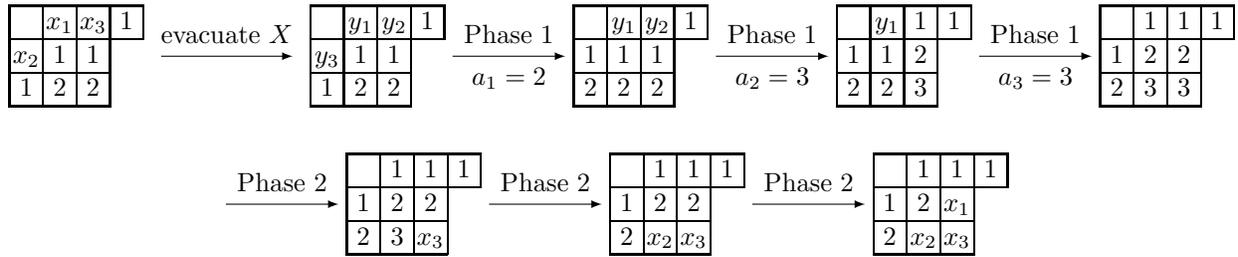
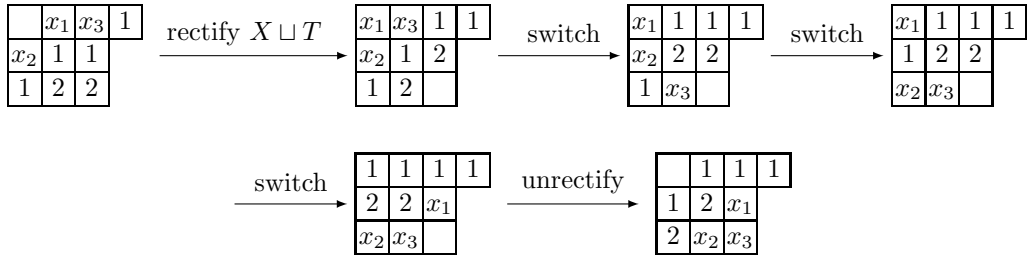
\begin{example}\label{ex:type-a-hop-coswitch}
    \[
    \begin{tikzpicture}
        \node at (0,0) {
        \begin{ytableau}
            \nonumber & x_1 & x_3 & 1 \\
            x_2 & 1 & 1\\
            1 & 2 & 2
        \end{ytableau}};
        \node at (2,.3) {$\mathrm{evacuate}\text{ } X$};
        \draw[-{latex}] (1.125,0) -- ++(1.75,0);
        \node at (4,0) {\begin{ytableau}
            \nonumber & y_1 & y_2 & 1\\
            y_3 & 1 & 1\\
            1 & 2 & 2
        \end{ytableau}};
        \node at (5.75,.3) {Phase 1};
        \node at (5.75,-.3) {$a_1=2$};
        \draw[-{latex}] (5,0) -- ++(1.5,0);
        \node at (7.5,0) {\begin{ytableau}
            \nonumber & y_1 & y_2 & 1 \\
            1 & 1 & 1\\
            2 & 2 & 2
        \end{ytableau}};
        \node at (9.25,.3) {Phase 1};
        \node at (9.25,-.3) {$a_2=3$};
        \draw[-{latex}] (8.5,0) -- ++(1.5,0);
        \node at (11,0) {\begin{ytableau}
            \nonumber & y_1 & 1 & 1\\
            1 & 1 & 2\\
            2 & 2 & 3
        \end{ytableau}};
        \node at (12.75,.3) {Phase 1};
        \node at (12.75,-.3) {$a_3=3$};
        \draw[-{latex}] (12,0) -- ++(1.5,0);
        \node at (14.5,0) {\begin{ytableau}
            \nonumber & 1 & 1 & 1\\
            1 & 2 & 2\\
            2 & 3 & 3
        \end{ytableau}};
    \end{tikzpicture}
\]
\[
\begin{tikzpicture}
        \draw[-{latex}] (4.5,0) -- ++(1.5,0);
        \node at (7,0) {\begin{ytableau}
            \nonumber & 1 & 1 & 1\\
            1 & 2 & 2\\
            2 & 3 & x_3
        \end{ytableau}};
        \node at (5.25,.3) {Phase 2};
        \draw[-{latex}] (8,0) -- ++(1.5,0);
        \node at (10.5,0) {\begin{ytableau}
            \nonumber & 1 & 1 & 1\\
            1 & 2 & 2\\
            2 & x_2 & x_3
        \end{ytableau}};
        \node at (8.75,.3) {Phase 2};
        \draw[-{latex}] (11.5,0) -- ++(1.5,0);
        \node at (14,0) {\begin{ytableau}
            \nonumber & 1 & 1 & 1\\
            1 & 2 & x_1\\
            2 & x_2 & x_3
        \end{ytableau}};
        \node at (12.25,.3) {Phase 2};
    \end{tikzpicture}
\]
\end{example}

\begin{example}
    \[
    \begin{tikzpicture}
        \node at (0,0) {
        \begin{ytableau}
            \nonumber & x_1 & x_3 & 1 \\
            x_2 & 1 & 1\\
            1 & 2 & 2
        \end{ytableau}};
        \node at (2.25,.3) {rectify $X\sqcup T$};
        \draw[-{latex}] (1.125,0) -- ++(2.5,0);
        \node at (4.625,0) {\begin{ytableau}
            x_1 & x_3 & 1 & 1\\
            x_2 & 1 & 2\\
            1 & 2 & \nonumber
        \end{ytableau}};
        \node at (6.5,.3) {$\shuffle$};
        \draw[-{latex}] (5.625,0) -- ++(1.625,0);
        \node at (8.25,0) {\begin{ytableau}
            x_1 & 1 & 1 & 1\\
            x_2 & 2 & 2\\
            1 & x_3 & \nonumber
        \end{ytableau}};
        \node at (10,.3) {$\shuffle$};
        \draw[-{latex}] (9.25,0) -- ++(1.5,0);
        \node at (11.75,0) {\begin{ytableau}
            x_1 & 1 & 1 & 1\\
            1 & 2 & 2\\
            x_2 & x_3 & \nonumber
        \end{ytableau}};
    \end{tikzpicture}
\]
\[
    \begin{tikzpicture}
        \node at (.75,.3) {$\shuffle$};
        \draw[-{latex}] (0,0) -- ++(1.5,0);
        \node at (2.5,0) {\begin{ytableau}
            1 & 1 & 1 & 1\\
            2 & 2 & x_1\\
            x_2 & x_3 & \nonumber
        \end{ytableau}};
        \node at (4.5,.3) {unrectify};
        \draw[-{latex}] (3.625,0) -- ++(1.75,0);
        \node at (6.5,0) {\begin{ytableau}
            \nonumber & 1 & 1 & 1 \\
            1 & 2 & x_1 \\
            2 & x_2 & x_3
        \end{ytableau}};
    \end{tikzpicture}
\]
\end{example}

We now describe an alternate local method for computing partial evacuation shuffling using the crystal operators described in Section \ref{sec:coplactic-operators}. Note that because we use crystal operators and relabeling steps which preserve the standardization of the tableau, the crystal algorithm is clearly coplactic. This version has the advantage that Phase 2 is easier to state than in the hopping algorithm.

\begin{definition}\label{def:type-A-crystal-alg}
    (Type A crystal algorithm) Let $X$ be a standard Young tableau such that $|X|=n$, and let $T$ be a Littlewood-Richardson semistandard Young tableau.

    Then the crystal algorithm in Type A is defined as follows:

    Replace the entries of $X$ in reading order by $-(n-1),\ldots,0$. Set $j=0$ and $i=n$.

    \begin{itemize}
        \item \textbf{Phase 1a:} If a $j+1$ precedes the singular $j$ in reading order, apply $E_{j}$ a number of times equal to one less than the number of $(j+1)$'s. Increment $j$ by one and repeat Phase 1a. Otherwise, record $a_{n-i+1}=j+1$, increment every entry less than $j+1$ by one, and proceed to Phase 1b.
        \item \textbf{Phase 1b:} If $i=1$, proceed to Phase 2. Otherwise, decrement $i$ by one, reset $j$ to be zero, and return to Phase 1a.
        \item \textbf{Phase 2:} If $i=n$, we are done. Otherwise, apply $F_{a_{n-i+1}},\ldots,F_{n}$. Replace the new $n+1$ entry by $x_{n-i+1}$. Increment $i$ by one and repeat Phase 2.
    \end{itemize}
\end{definition} 

We demonstrate this algorithm in the following example. 

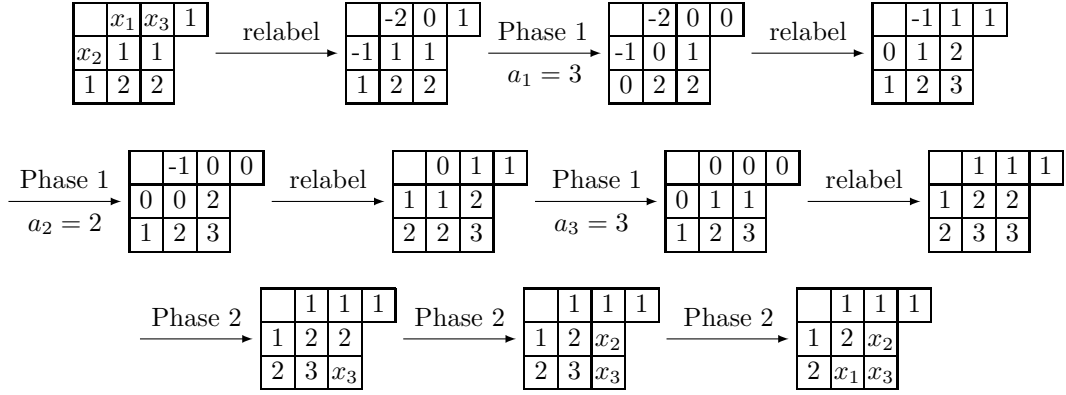
\begin{example}
    \[
    \begin{tikzpicture}
        \node at (0,0) {
        \begin{ytableau}
            \nonumber & x_1 & x_3 & 1 \\
            x_2 & 1 & 1\\
            1 & 2 & 2
        \end{ytableau}};
        \node at (1.875,.3) {relabel};
        \draw[-{latex}] (1,0) -- ++(1.625,0);
        \node at (3.625,0) {\begin{ytableau}
            \nonumber & $-2$ & $0$ & 1 \\
            $-1$ & 1 & 1\\
            1 & 2 & 2
        \end{ytableau}};
        \node at (5.35,.3) {Phase 1};
        \node at (5.375,-.3) {$a_1=3$};
        \draw[-{latex}] (4.625,0) -- ++(1.5,0);
        \node at (7.125,0) {\begin{ytableau}
            \nonumber & $-2$ & 0 & 0\\
            $-1$ & 0 & 1\\
            0 & 2 & 2
        \end{ytableau}};
        \node at (8.875,.3) {relabel};
        \draw[-{latex}] (8.125,0) -- ++(1.5,0);
        \node at (10.625,0) {\begin{ytableau}
            \nonumber & $-1$ & 1 & 1\\
            0 & 1 & 2\\
            1 & 2 & 3
        \end{ytableau}};
    \end{tikzpicture}
\]
\[
\begin{tikzpicture}
        \node at (-.25,.3) {Phase 1};
        \node at (-.25,-.3) {$a_2=2$};
        \draw[-{latex}] (-1,0) -- ++(1.5,0);
        \node at (1.5,0) {
        \begin{ytableau}
            \nonumber & $-1$ & 0 & 0 \\
            0 & 0 & 2\\
            1 & 2 & 3
        \end{ytableau}};
        \node at (3.25,.3) {relabel};
        \draw[-{latex}] (2.5,0) -- ++(1.5,0);
        \node at (5,0) {\begin{ytableau}
            \nonumber & 0 & 1 & 1\\
            1 & 1 & 2\\
            2 & 2 & 3
        \end{ytableau}};
        \node at (6.825,.3) {Phase 1};
        \node at (6.75,-.3) {$a_3=3$};
        \draw[-{latex}] (6,0) -- ++(1.625,0);
        \node at (8.625,0) {\begin{ytableau}
            \nonumber & 0 & 0 & 0 \\
            0 & 1 & 1\\
            1 & 2 & 3
        \end{ytableau}};
        \node at (10.375,.3) {relabel};
        \draw[-{latex}] (9.625,0) -- ++(1.5,0);
        \node at (12.125,0) {\begin{ytableau}
            \nonumber & 1 & 1 & 1\\
            1 & 2 & 2\\
            2 & 3 & 3
        \end{ytableau}};
    \end{tikzpicture}
\]
\[
\begin{tikzpicture}
        \node at (-.25,.3) {Phase 2};
        \draw[-{latex}] (-1,0) -- ++(1.5,0);
        \node at (1.5,0) {
        \begin{ytableau}
            \nonumber & 1 & 1 & 1 \\
            1 & 2 & 2\\
            2 & 3 & x_3
        \end{ytableau}};
        \node at (3.25,.3) {Phase 2};
        \draw[-{latex}] (2.5,0) -- ++(1.5,0);
        \node at (5,0) {\begin{ytableau}
            \nonumber & 1 & 1 & 1\\
            1 & 2 & x_2\\
            2 & 3 & x_3
        \end{ytableau}};
        \node at (6.825,.3) {Phase 2};
        \draw[-{latex}] (6,0) -- ++(1.625,0);
        \node at (8.625,0) {\begin{ytableau}
            \nonumber & 1 & 1 & 1 \\
            1 & 2 & x_2\\
            2 & x_1 & x_3
        \end{ytableau}};
    \end{tikzpicture}
\]
\end{example}

\begin{theorem} \label{thm: hop=crystal}
    The output of $\crystal(X,T)$ and the output of $\hop(X,T)$ agree.
\end{theorem}

\begin{proof}
    In $\crystal(X,T)$, at each application of $E_j$, we isolate the $j+1$ entry that precedes the unique $j$ and lower all other $(j+1)$'s to $j$'s. This makes the $j+1$ that preceded $j$ our new unique number. This has the same effects as directly switching $y_i$ with the $j$ preceding it in reading order in $\hop(X,T)$. The only difference is that in $\crystal(X,T)$, our entries have been lowered, but we correct this by incrementing at the end of Phase 1a, so Phase 1a from both $\crystal(X,T)$ and $\hop(X,T)$ agree. Phase 1b is the same for both algorithms, so we need only consider whether their Phase 2's agree. 

    In Phase 2 of $\crystal(X,T)$, we start by applying $F_{a_i}$. This raises the first $j$ entry whose $(j,j+1)$ suffix is tied, as that will be the last unmatched $j$. In $\hop(X,T)$, we replaced the first $j$ in reading order by $x_{n-i+1}$. Swapping $x_{n-i+1}$ with this same $j$ entry with a tied suffix puts a $j$ back where $x_{n-i+1}$ was and isolates the same entry that $\crystal(X,T)$ raised to be a $j+1$. We repeat this process in both $\crystal(X,T)$ and $\hop(X,T)$ until the same final entry is $x_{n-i+1}$, so Phase 2 agrees for both algorithms.
\end{proof}

\begin{remark}
    We may compute Phase 1 and Phase 2 using either of the descriptions above. In general, a hopping Phase 1 and crystal Phase 2 is the most computationally efficient and practical algorithm.
\end{remark}

\begin{definition}
    (Type A mixed algorithm) Let $X$ be a standard Young tableau such that $|X|=n$, and let $T$ be a Littlewood-Richardson semistandard Young tableau. 
    
    Then the mixed algorithm in Type A is defined as follows: 

    Set $i=n$ and $j=1$.

    \begin{itemize}
        \item \textbf{Phase 1a:} If $x_i$ precedes all of the $j$s in reading order, replace $x_i$ by $j$, record $a_{n-i+1}=j$, and proceed to Phase 1b. If $x_i$ does not precede all of the $j$s in reading order, switch $x_i$ with the $j$ immediately preceding it in reading order. Increment $j$ and repeat Phase 1a.

        \item \textbf{Phase 1b:} If $i=1$, proceed to Phase 2. If $i\neq 1$, decrement $i$, reset $j=1$, and repeat Phase 1a.

        \item \textbf{Phase 2:} Apply $F_{a_{n-i+1}},\ldots,F_{n}$. Replace the new $n+1$ entry by $x_{n-i+1}$. Increment $i$ by one and repeat Phase 2.
    \end{itemize}
\end{definition}

\section{Related Results}\label{sec:related-results}

\subsection{Variations of the algorithms}

In this section, we state several variations of our local algorithms. We begin with a statement connecting our local algorithms to $\alg{}$ defined in Section \ref{sec:switching-ops}.

\begin{corollary}\label{cor:pesh-recovery}
    In either algorithm, the output we get is $\alg(X,T)$.
\end{corollary}

\begin{proof}
    Let $\simpalg(X,T)$ be either local algorithm applied to $(X,T)$. We proved earlier that
$$\coswitch(X,T)=\simpalg(\coplacticEvac(X),T).$$
    We then use the fact that coplactic evacuation is an involution to rewrite the statement above:
$$\coswitch(\coplacticEvac(X),T)=\simpalg(\coplacticEvac(\coplacticEvac(X)),T)=\simpalg(X,T).$$
    Notice that this is exactly the definition of $\alg$, meaning the output of $\simpalg$ must agree with that of $\alg$
$$\simpalg(X,T)=\coswitch(\coplacticEvac(X),T)=\alg(X,T)$$ and the proof is complete.
\end{proof}

When computing $\coswitch(X,T)$, both the hopping and crystal algorithms include a single non-local step, which is the coplactic evacuation of the inner tableau $X$. However, we may use the fact that $\simpalg{}$ computes $\alg{}$ to construct a completely local version of each coplactic bijection described in Section \ref{sec:switching-ops} by applying the following Lemma:

\begin{lemma}
    Both $\coswitch{}$ing and $\mathrm{evacuation}$ can be written in terms of $\alg$ as follows:
    \begin{itemize}
        \item $\coplacticEvac(X,T)$ is the result of splitting $\alg(X\sqcup T,\emptyset)$.
        \item $\coswitch(X,T)=\alg(\alg(X,\emptyset),T)$
    \end{itemize}
\end{lemma}

\begin{proof}
    The first statement follows from the fact that $\coswitch{}$ has no effect when one of of the arguments is empty. The second statement simplifies to $\coswitch(X,T)=\alg(\coplacticEvac(X),T)$, which follows from the fact that $\coplacticEvac{}$ is an involution.
\end{proof}

Hence, we have local algorithms which compute each of the coplactic switching operations defined in Section \ref{sec:switching-ops}.

We also observe that crystal operators, and consequently the switches of the hopping algorithm, commute under certain conditions. This gives us the following statement which staggers the steps of Phase 1:

\begin{remark}\label{remark_stagger}
    Because they only depend on the $j$-subword in $T$, we may stagger the steps of Phase 1. That is, we may perform a single step of Phase 1 on $x_n$ through some index $j_n$, then perform another step on $x_{n-1}$ through the index $j_{n-1}\leq i_n$, and so on for all entries $x_i$ that are currently in Phase 1.
\end{remark}

While this means the order we apply the Phase 1 hops or crystal operators is very easy to modify, Phase 2 is more challenging since the steps depend on the $(j,j+1)$-subword and the entries of $X$ do not necessarily begin at the same index. Hence, we do not have a similar statement which alters the order in which we perform the steps of Phase 2.

Another feature of our algorithm is that the inner tableau, $X$, gets absorbed into $T$ during Phase 1. While we seem to be losing data on $X$, which the transition data records in some capacity, we may alter the steps of the hopping algorithm to keep track of the entries of $X$ as we perform compute the local algorithms.

\begin{remark}[Label tracking]
    
    To track the entries of $X$ after Phase 1, apply the following modification to the statement of Definition \ref{def:type_a_hop}: whenever we say to replace some $x_i$ entry with a $j$ entry, instead replace it with the symbol $j_{n+i-1}$, i.e., a subscripted $j$. Thereafter, there is always a unique $i$ entry with subscript $n+i-1$ denoting the position of $x_{n+i-1}$ in the tableau. Then, at the corresponding step of Phase 2 we turn the entry with subscript $(n+i-1)$ into $x_{n+i-1}$.

    The main detail to check here is that $j_{n+i-1}$ is actually the smallest in reading order when beginning the corresponding step of Phase 2. For this we observe that when $y_i$ completes Phase 1 and is replaced with a $j_{n+i-1}$, it must be the smallest $j$ in standardization order since there are no $j$ entries to the left.
    
    Notice that the steps of both Phase 1 and Phase 2 switch $x_k$ with an the closest $j$ entry in reading order, meaning the value of each marked $j$ entry in standardization order is not affected by the hops. Hence, the only possible entries smaller than $j_{n+i-1}$ are of the form $j_{n+k-1}$, where $k>j$. However, we slide out all entries with larger subscripts first, meaning $j_{n+i-1}$ must be the smallest $j$ in standardization order.
\end{remark}

We demonstrate label tracking below using the same tableaux as Example \ref{ex:type-a-hop}.

\begin{example}
    \[
    \begin{tikzpicture}
        \node at (0,0) {
        \begin{ytableau}
            \nonumber & x_1 & x_3 & 1 \\
            x_2 & 1 & 1\\
            1 & 2 & 2
        \end{ytableau}};
        \node at (2,.3) {Phase 1};
        \node at (2,-.3) {$a_1=3$};
        \draw[-{latex}] (1.25,0) -- ++(1.5,0);
        \node at (4,0) {\begin{ytableau}
            \nonumber & x_1 & 1 & 1 \\
            x_2 & 1 & 2\\
            1 & 2 & 3_1
        \end{ytableau}};
        \node at (5.75,.3) {Phase 1};
        \node at (5.75,-.3) {$a_2=2$};
        \draw[-{latex}] (5,0) -- ++(1.5,0);
        \node at (7.5,0) {\begin{ytableau}
            \nonumber & x_1 & 1 & 1\\
            1 & 1 & 2\\
            2_2 & 2 & 3_1
        \end{ytableau}};
        \node at (9.25,.3) {Phase 1};
        \node at (9.25,-.3) {$a_3=3$};
        \draw[-{latex}] (8.5,0) -- ++(1.5,0);
        \node at (11,0) {\begin{ytableau}
            \nonumber & 1 & 1 & 1\\
            1 & 2 & 2\\
            2_2 & 3_3 & 3_1
        \end{ytableau}};
    \end{tikzpicture}
\]
\[
\begin{tikzpicture}
        \draw[-{latex}] (4.5,0) -- ++(1.5,0);
        \node at (7,0) 
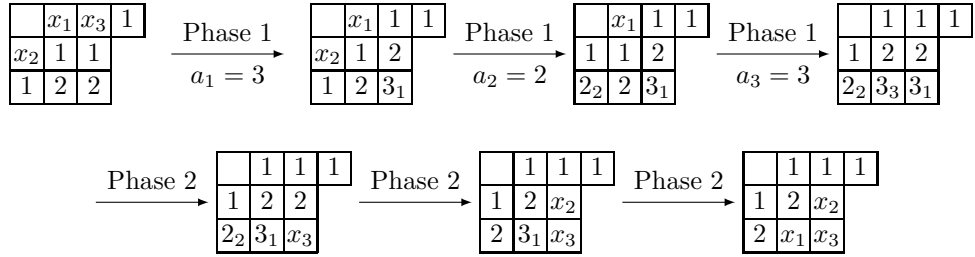
{\begin{ytableau}
            \nonumber & 1 & 1 & 1\\
            1 & 2 & 2\\
            2_2 & 3_1 & x_3
        \end{ytableau}};
        \node at (5.25,.3) {Phase 2};
        \draw[-{latex}] (8,0) -- ++(1.5,0);
        \node at (10.5,0) {\begin{ytableau}
            \nonumber & 1 & 1 & 1\\
            1 & 2 & x_2\\
            2 & 3_1 & x_3
        \end{ytableau}};
        \node at (8.75,.3) {Phase 2};
        \draw[-{latex}] (11.5,0) -- ++(1.5,0);
        \node at (14,0) {\begin{ytableau}
            \nonumber & 1 & 1 & 1\\
            1 & 2 & x_2\\
            2 & x_1 & x_3
        \end{ytableau}};
        \node at (12.25,.3) {Phase 2};
    \end{tikzpicture}
\]
\end{example}

Using Corollary \ref{cor:pesh-recovery}, we observe that the crystal and hopping algorithms intrinsically compute \algorithm; we simply compute coswitching by introducing an evacuation step before or after applying the operation. We now give one final variant of local coswitching that is computed entirely in terms of hops, without explicitly evacuating either tableau. This variant is given by changing the order of the boxes of $X$ after Phase 1, before proceeding to Phase 2. To do so, we instead perform an auxiliary ``evacuation''-like calculation on the transition data (considered as a word of integers).

To describe this variant we require the following lemma, which allows us to recover the slide class of the inner tableau from the transition data:

\begin{lemma}\label{lemma:transition-data-reading-word}
    Let $X$ be a semistandard Young tableau and $T$ be a Littlewood-Richardson tableau extending $X$. Let the pair $(T',X')$ be the result of rectifying and switching $(X,T)$. Then the reading word of $X'$ may be computed using the transition data as follows:

    \begin{itemize}
        \item Compute the transition data $a=(a_1,\ldots,a_n)$ associated with either $\hop(X,T)$ or $\crystal(X,T)$, recording it in a two-line array $\begin{bmatrix}
            a_1 & \ldots & a_n\\
            1 & \ldots & n
        \end{bmatrix}$.
        \item Sort the vertical pairs of the array largest to smallest by the entries of the top row, where ties are broken by entries in the bottom row such that if $i<j$, the column containing $i$ precedes the column containing $j$.
    \end{itemize}

    The reading word of $X'$ is precisely the word written in the second row of the array.
\end{lemma}


\begin{proof}
    We begin by reducing to the case where $X$ is rectified, by coplacticity of the transition data and the construction of $X'$. Recall that the transition data is equivalent to the row data associated with $\hop(X,T)$ and $\crystal(X,T)$ (Theorem \ref{thm: hop=esh}). In particular, when $X$ is rectified, this means $a_i$ is the row containing $x_i$ in $X'$. We encode this information in our array as vertical pairs $\begin{bmatrix}
        a_i\\
        i
    \end{bmatrix}$.
    
    The reading word of $X'$ is obtained by reading off the entries of $X'$ from left to right, starting from the bottom row. So, the first entries in reading order are those with the largest associated value in the transition data, followed by the entries with the next largest value, and so on.

    Once we have determined which entries are contained on each row, we use the fact that $x_i$ is the $i$-th smallest entry of $X'$ in standardization order to complete the word: suppose $i_1<\ldots<i_k$ and $x_{i_1},\ldots,x_{i_k}$ are all contained in the same row, i.e. have the same corresponding values in the transition data. Because $X'$ is a semistandard Young tableau, $x_{i_1}$ must come before $x_{i_2}$, which is before $x_{i_3}$, and so on. Observe that this construction of the reading word of $X'$ is exactly encoded in the array sorting rules described above. Hence, the second row of the sorted array must be the reading word of $X'$.
\end{proof}


We note that $\alg(X, T)$ and $\coswitch(X, T)$ give the same $T'$ and differ only in that $\alg(X, T)$ yields a different tableau $X''$. By Lemma \ref{lemma:transition-data-reading-word}, we can compute the reading word of $X''$ in the middle of the algorithm via the transition data. Thus, we want to modify the transition data so that the remainder of the algorithm computes $X'$ rather than $X''$.


\begin{remark}
    Let $S$ be the tableau obtained immediately after Phase 1 and before beginning Phase 2 of $\hop(X,T)$. Because crystal operators preserve the dual equivalence class of a tableau and $S$ is a Littlewood-Richardson tableau with shape $\shape(X\sqcup T)$, we know that $S$ is the highest weight representative for the dual equivalence class of $X\sqcup T$. In particular, $S$ does not depend on the slide class of $X$. Furthermore, we see that $\weight(S)$ is the vector sum of the transition data and $\weight(T)$.
\end{remark}


\begin{proposition}\label{prop:array algorithm} (Array algorithm)
    Let $X$ be a semistandard Young tableau and $T$ be a Littlewood-Richardson tableau extending $X$. Then $\coswitch(X,T)$ may be computed by the following:

    \begin{itemize}
        \item Perform Phase 1 of either the hopping or crystal statement of $\simpalg(X,T)$, recording the transition data $a=(a_1,a_2,\ldots, a_n)$ in a two-line array
        $
        \begin{bmatrix}
            a_1 & \ldots & a_n\\
            1 & \ldots& n
        \end{bmatrix}.$
        \item Sort the array as described in Lemma \ref{lemma:transition-data-reading-word}, then evacuate the second row, treating its entries as a word. Then sort the vertical pairs of the resulting array smallest to largest by the entries of the bottom row, giving us the array $\begin{bmatrix}
            b_1 & \ldots & b_n\\
            1 & \ldots& n
        \end{bmatrix}.$
        \item Perform Phase 2 of either local algorithm using the new transition data $b=(b_1,\ldots,b_n)$ in place of $a$.


        
        
    \end{itemize}
\end{proposition}

\begin{proof}

This description amounts to replacing the steps of Phase 2 by the steps of Phase 2 occurring during the calculation of $\simpalg(\coplacticEvac(X), T)$. (Note that the combined tableau of shape $X \sqcup T$ obtained at the end of Phase 1 depends only on the dual equivalence classes of $X$ and $T$, which are preserved by evacuation.) 

In particular, it suffices to verify that the vector $(b_1, \ldots, b_n)$ defined above equals the transition data corresponding to $\simpalg(\coplacticEvac(X), T)$. By coplacticity, we may simplify to the rectified setting. 

Let $(T',X')$ be the result of coswitching the rectified pair $(X,T)$. Using Corollary \ref{cor:pesh-recovery} and Lemma \ref{lemma:transition-data-reading-word}, we know that the first sorting operation gives the reading word of $e(X')$ in the second row. Then we evacuate the word, which results in the reading word of $X'$. Let $b=(b_1,\ldots, b_n)$ be as defined above. We observe that, given a fixed weight determined by $\shape(X')$, $b$ is the unique vector such that the array 
$\begin{bmatrix}
            b_1 & \ldots & b_n\\
            1 & \ldots& n
        \end{bmatrix}$ 
yields the reading word of $X'$ after sorting the array as described in Lemma \ref{lemma:transition-data-reading-word}. Lemma \ref{lemma:transition-data-reading-word} also guarantees that the transition data associated with $\simpalg(\coplacticEvac(X),T)$ has this property, so it must be equal to $b$ by uniqueness of the vector.
\end{proof}





\begin{example}
Compared to Example \ref{ex:type-a-hop-coswitch}, Phase 1 of the example below differs slightly in the order we perform the hops, but we still end up with the same intermediate tableau right before Phase 2. Then, we recover the transition data used in Example \ref{ex:type-a-hop-coswitch}, before performing the exact same sequence of hops.

    
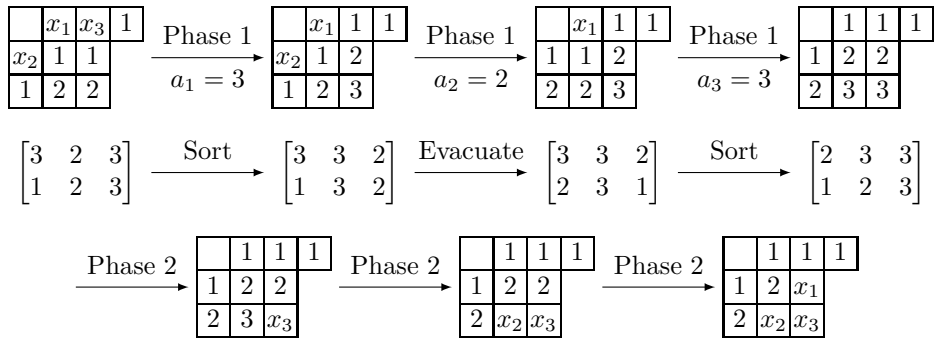
\[
    \begin{tikzpicture}
        \node at (0,0) {
        \begin{ytableau}
            \nonumber & x_1 & x_3 & 1 \\
            x_2 & 1 & 1\\
            1 & 2 & 2
        \end{ytableau}};
        \node at (1.75,.3) {Phase 1};
        \node at (1.75,-.3) {$a_1=3$};
        \draw[-{latex}] (1,0) -- ++(1.5,0);
        \node at (3.5,0) {\begin{ytableau}
            \nonumber & x_1 & 1 & 1\\
            x_2 & 1 & 2\\
            1 & 2 & 3
        \end{ytableau}};
        \node at (5.25,.3) {Phase 1};
        \node at (5.25,-.3) {$a_2=2$};
        \draw[-{latex}] (4.5,0) -- ++(1.5,0);
        \node at (7,0) {\begin{ytableau}
            \nonumber & x_1 & 1 & 1\\
            1 & 1 & 2\\
            2 & 2 & 3
        \end{ytableau}};
        \node at (8.75,.3) {Phase 1};
        \node at (8.75,-.3) {$a_3=3$};
        \draw[-{latex}] (8,0) -- ++(1.5,0);
        \node at (10.5,0) {\begin{ytableau}
            \nonumber & 1 & 1 & 1\\
            1 & 2 & 2\\
            2 & 3 & 3
        \end{ytableau}};
    \end{tikzpicture}
\]
\[
\begin{tikzpicture}
    \node at (0,0) {
        $\begin{bmatrix}
            3 & 2 & 3\\
            1 & 2 & 3
        \end{bmatrix}$};
    \node at (1.75,.3) {Sort};
    \draw[-{latex}] (1,0) -- ++(1.5,0);
    \node at (3.5,0) {
        $\begin{bmatrix}
            3 & 3 & 2\\
            1 & 3 & 2
        \end{bmatrix}$};
    \node at (5.25,.3) {Evacuate};
    \draw[-{latex}] (4.5,0) -- ++(1.5,0);
    \node at (7,0) {
        $\begin{bmatrix}
            3 & 3 & 2\\
            2 & 3 & 1
        \end{bmatrix}$};
    \node at (8.75,.3) {Sort};
    \draw[-{latex}] (8,0) -- ++(1.5,0);
    \node at (10.5,0) {
        $\begin{bmatrix}
            2 & 3 & 3\\
            1 & 2 & 3
        \end{bmatrix}$};
\end{tikzpicture}
\]
\[
\begin{tikzpicture}
        \draw[-{latex}] (4.5,0) -- ++(1.5,0);
        \node at (7,0) {\begin{ytableau}
            \nonumber & 1 & 1 & 1\\
            1 & 2 & 2\\
            2 & 3 & x_3
        \end{ytableau}};
        \node at (5.25,.3) {Phase 2};
        \draw[-{latex}] (8,0) -- ++(1.5,0);
        \node at (10.5,0) {\begin{ytableau}
            \nonumber & 1 & 1 & 1\\
            1 & 2 & 2\\
            2 & x_2 & x_3
        \end{ytableau}};
        \node at (8.75,.3) {Phase 2};
        \draw[-{latex}] (11.5,0) -- ++(1.5,0);
        \node at (14,0) {\begin{ytableau}
            \nonumber & 1 & 1 & 1\\
            1 & 2 & x_1\\
            2 & x_2 & x_3
        \end{ytableau}};
        \node at (12.25,.3) {Phase 2};
    \end{tikzpicture}
\]
\end{example}

\subsection{Reversing the algorithms}

Despite the fact that $\coswitch$ing is an involution, every algorithm so far relies on the outer tableau being Littlewood-Richardson. Hence, we require a description of the inverse our algorithms in order to compute $\coswitch$ing when the inner tableau is highest weight and the outer tableau is an arbitrary semistandard tableau. We now describe various local methods for computing $\coswitch^{-1}(T,X)$, beginning with the reverse of the hopping algorithm.

\begin{definition}
(Type A reverse hopping algorithm) Let $d$ be the largest numerical entry of $T\sqcup X$. Set $j=d+1$ and $i=1$.
\begin{itemize}
   \item \textbf{Reverse Phase 2b:} If the $(j-1,j)$ suffix of $x_i$ contains strictly more $(j-1)$'s than $j$'s, replace $x_i$ with $j$, record $a_i=j$, and proceed to Reverse Phase 2a. If not, look at the $(j-2,j-1)$ suffix of $x_i$. If it is tied, do not move $x_i$. Otherwise, switch $x_i$ with the last $j-1$ in reading order that would make the $(j-2,j-1)$ suffix of $x_i$ tied. If no such entry exists, switch $x_i$ with the first $j-1$ in reading order. Decrement $j$ and repeat Reverse Phase 2b with the same $i$ value.

   \item \textbf{Reverse Phase 2a:} If $i=n$, reset $i=1$ and proceed to Reverse Phase 1b. If not, increment $i$, reset $j=d+1$, and return to Reverse Phase 2b.

   \item \textbf{Reverse Phase 1b:} Check the $j$ value to which $a_i$ is equal. Replace the first $j$ in reading order with $x_i$ and proceed to Reverse Phase 1a.

   \item \textbf{Reverse Phase 1a:} Swap $x_i$ with the first $j-1$ after it in reading order. Decrement $j$. If $j\neq0$, repeat Reverse Phase 1a. If $j=0$, increment $i$. If $i\neq0$, return to Reverse Phase 1b.
\end{itemize}
\end{definition}

We now give a description of the reverse crystal algorithm. Note that the crystal algorithm is relatively simple to invert since $E_i$ and $F_i$ are partial inverses of each other, so the main detail is in detecting when an entry finishes Reverse Phase 2.

\begin{definition}
(Type A reverse crystal algorithm) Let $d$ be the largest numerical entry of $T\sqcup X$. Set $i=1$.
\begin{itemize}
   \item \textbf{Reverse Phase 2:} Replace $x_i$ with $d+1$. Apply $E_d,\ldots,E_{j}$, where $E_{j-1}$ is undefined. Record $a_i=j$, and increment $i$. If $i=n+1$, proceed to Reverse Phase 1b. Otherwise, repeat Reverse Phase 2.

   \item \textbf{Reverse Phase 1b:} Reset $i=1$ and proceed to Reverse Phase 1a.

   \item \textbf{Reverse Phase 1a:} If $a_i=1$, Replace the first 1 in reading order by $x_i$. Otherwise, lower every numerical entry less than $a_i$ by one. Check the $j$ value to which $a_i$ is equal, and replace the first $j$ in reading order with $j-1$. Apply $F_{j-2},\ldots,F_0$ where each $F_k$ is applied a number of times equal to the number of $k$'s. Replace the new unique 0 by $x_i$. Increment $i$. If $i\neq0$, repeat Reverse Phase 1a with the new $i$ value.
\end{itemize}
\end{definition}

\begin{theorem}
    The output of $\mathrm{revcrystal}(\crystal(X,T))=(X,T)$.
\end{theorem}

\begin{proof}
    In $\crystal(X,T)$, we changed the largest numerical entry to $x_i$ ending with $i=n$, so replacing $x_i$ with an entry one larger than $d$ beginning with $i=n$ will certainly undo the replacement from Phase 2 of $\crystal(X,T)$. Additionally, since crystal raising and crystal lowering operators are known inverses, applying $E_d,\ldots,E_{j}$ will undo the $F_{a_{n-i+1}},\ldots,F_n$ from Phase 2 of $\crystal(X,T)$. Since $j-1$ is the last place $E_i$ is defined and $a_{n-i+1}$ was the first place $F_i$ was defined, our chains are the same length. Repeating these steps for $1,\ldots,n$ will undo the steps from Phase 2 of $\crystal(X,T)$ that were applied for $n,\ldots,1$.

    In Reverse Phase 1a, lowering every entry less than $a_i$ undoes the raising from Phase 1a of $\crystal(X,T)$. The crystal lowering operators applied in this phase undo the operators applied in Phase 1a of $\crystal(X,T)$ since we are applying the same chain of operators where each one is applied the same number of times. 
\end{proof}

\begin{theorem}
    The output of $\mathrm{revcrystal}(T,X)$ and $\mathrm{revhop}(T,X)$ agree.
\end{theorem}

\begin{proof}
    In Reverse Phase 2b of $\mathrm{revhop}(T,X)$, if the the suffix of $x_i$ has strictly more $(j-1)$'s than $j$'s, then if $x_i$ was replaced by $j$ as in $\mathrm{revcrystal}(T,X)$, $E_{j-1}$ would be undefined because that $j$ would be paired with a $j-1$. Thus, if this suffix condition is met, we move on as if the chain of raising operators in Reverse Phase 2 of $\mathrm{revcrystal}(T,X)$ was completed, matching the transition between phases in $\mathrm{revhop}$.

    If the suffix of $x_i$ does not contain strictly more $(j-1)$'s than $j$'s, then $E_j$ will be defined and changes $x_i=j$ to $j-1$.  Looking ahead to trying to apply $E_{j-1}$, we can do so if there is a (rightmost) suffix in the new word for which $(j-2,j-1)$ is tied, and the entry we change corresponds to the box we switch $x_i$ with in the hopping algorithm.

    In Reverse Phase 1a of $\mathrm{revcrystal}(T,X)$, at each application of a lowering operator, we are isolating a single entry. When we lower all entries less than $j+1$ by 1, we turn all $j$'s into $(j-1)$'s, so replacing the first $j+1$ in reading order by a $j$ then guarantees we have a single $j$ entry. This $j$ will stay isolated throughout Reverse Phase 1a because it will match with the $j-1$ immediately after it in reading order and all other $(j-1)$'s will become $j$'s, leaving only a single $j-1$ behind to repeat this process on. This isolation process directly corresponds to hopping $x_i$ through our tableau in Reverse Phase 1a of $\mathrm{revhop}(T,X)$.
\end{proof}

Each algorithm defined in the paper so far relies on at least one tableau being reverse-ballot; however, it is possible to translate all of the statements and proofs to the case where $T$ is the lowest weight representative for its dual equivalence class:

\begin{remark}
    If $X$ is an arbitrary tableau extending the shape of a \textit{lowest} weight tableau $T$, we may compute $\esh^{-1}(T,X)$ by performing a reflected version of any of the algorithms described above. For example, the mirrored crystal algorithm turns all $E_j$ crystal operators into $F_j$ operators and vice versa, then it reverses the indexing to account for $X$ being the outer tableau.
\end{remark}

\subsection{Transition data are \texorpdfstring{$\lambda$}{lambda}-dominant words}

We now describe another property of the transition data. In particular, we want to understand what types of vectors we can expect for the transition data when running these local algorithms.

As discussed before, Phase 1 ends with the construction of the Littlewood-Richardson representative $S$ for the dual equivalence class of $X\sqcup T$. However, given $S$, it is not obvious which vectors arise as transition data leading to $S$, i.e. correspond to well defined sequences of hops or crystal operators. For this, we introduce a condition which says the weight of the transition data must be `compatible' with the rectification shape of $T$ as described below:

\begin{definition}
    Let $\lambda$ be a partition. A word $w$ is \textbf{$\lambda$-dominant} if concatenating $w$ with any ballot word $v$ with weight $\lambda$ results in a ballot word $vw$.
\end{definition}

One use of $\lambda$-dominant words is in the enumeration of Littlewood-Richardson tableaux, see \cite[Page 66]{Fulton}. We also observe that $w$ is $\lambda$-dominant if there exists some ballot word $v$ with weight $\lambda$ such that $vw$ is ballot, which is straightforward to see using the definition of ballot.


\begin{proposition}\label{prop:lambda-dominant}
    Let $X$ be a semistandard Young tableau and $T$ be a Littlewood-Richardson tableau extending $X$. Suppose $T$ has rectification shape $\lambda$. Then the transition data obtained from Phase 1 of any of our local algorithms, when treated as a word $w=h(x_n)\ldots h(x_1)$, is $\lambda$-dominant.
\end{proposition}

\begin{proof}
    We may reduce to the case where $X$ is rectified since the transition data is coplactic.

    First, we prove that the transition data $h(X,\emptyset)=w$ is $\lambda$-dominant for all choices of $X$. This is equivalent to showing $w$ is ballot since $\lambda$ is empty.
    
    If $X$ is a standard Young tableau in which the entry $i$ is in row $r_i$, then the word $r = (r_1, \ldots, r_n)$ is ballot. It is well known that this gives a bijection between standard Young tableaux of shape $\lambda$ and ballot words of weight $\lambda$, see \cite[Page 176]{Sagan} for details.

    Accordingly, let $r = (r_1, \ldots, r_n)$ be the row data of $(X,\emptyset)$. We observe that coswitching has no effect when one of the arguments is empty, so $X$ must remain a rectified Young tableau and each $x_i$ must be in row $r_i$. As stated above, this means the row data must be a ballot word. Because $r=h(e(X),\emptyset)$, this also means the transition data is ballot when $T=\emptyset$.
    



    We now prove the main claim when $T$ is any Littlewood-Richardson tableau extending $X$. Suppose $X\sqcup T$ has shape $\mu/\emptyset$. Consider Phase 1 of $\hop(X\sqcup T,\emptyset)$. We know the transition data $h(X\sqcup T,\emptyset)=u$ is a ballot word with weight $\mu$ by the statement above. In particular, the first $|T|$ entries of $u$ must form a ballot word $v$ with weight $\lambda$ since this is equivalent to the steps of Phase 1 during the computation of $\hop(T,\emptyset)$.

    We also observe that the first $|T|$ steps of Phase 1 of $\hop(X\sqcup T,\emptyset)$ results in a ballot outer tableau dual equivalent to the $|T|$ largest entries of $X\sqcup T$, a consequence of Theorem \ref{thm: hop=crystal} and the fact that crystal operators preserve dual equivalence. This means we recover $T$ during Phase 1 of the local algorithm and the final $|X|$ steps of Phase 1 hop $X$ through $T$. Hence, the $|X|$ remaining entries of of $u$ are exactly the transition data of $\hop(X,T)$. This tells us that $w=h(X,T)$ is a word such that the concatenation $vw$ is a ballot. As noted before, if $vw$ is ballot for some ballot word $v$ with weight $\lambda$, then $w$ must be $\lambda$-dominant.
\end{proof}

\begin{proposition}
    Let $S$ be a ballot tableau with weight $\mu$ and $\lambda\subseteq\mu$. Then there is a bijection between pairs of tableau $(X,T)\in (\mathrm{SYT},\LR(\lambda))$ such that $X\sqcup T=_{\mathrm{DE}}S$ and $\lambda$-dominant words $w$ with weight $\mu-\lambda$.
\end{proposition}

\begin{proof}
    The forward direction is shown by \ref{prop:lambda-dominant}. 
    
    For the reverse direction, we begin by rectifying $S$. Then, we fill in the inner (straight) shape $\lambda$ with the unique Littlewood-Richardson tableau $T'$. We can also fill in the outer shape $\mu/\lambda$ with the standard tableau $X'$ using the given $\lambda$-dominant word $w$ in conjunction with Lemma \ref{lemma:transition-data-reading-word} (with computes the reading word of $X'$). Finally, we run the reverse hopping algorithm on $(T',X')$ an unrectify the resulting pair to obtain $(X,T)$, which has $w$ as the corresponding transition data.
\end{proof}

From this, we see that $\lambda$-dominance fully characterizes the set of possible vectors we may obtain as transition data when running the local algorithms.

\subsection{Geometric applications}

Although we plan to describe the geometric applications of our algorithm in a future paper, we want to briefly give examples of Gillespie and Levinson's geometric results from \cite[Proposition 6.3]{1box_a} in the general setting.

\begin{remark}
Let $\alpha,\lambda,\beta,\gamma$ be partitions such that their sizes sum to $k(n-k)$. Then the monodromy of $S(\alpha,\lambda,\beta,\gamma)$ (see section \ref{sec:geometry}) is described by $\omega(\lambda,\beta)=\shuffle\circ\esh(\lambda,\beta)$. A fixed point under monodromy occurs when $\esh$ and $\shuffle$ agree, since switching is an involution and this would force $\omega$ to compose to the identity. One class of fixed points is given by pairs of tableaux that, when rectified, have a column of $x_i$ values on their left with nothing falling below $x_n$. 
\end{remark}

\begin{example}\label{ex:monodromy}
We observe the behavior of the following chain of tableaux, as well as its rectification.

\begin{center}
    $X\sqcup T$:
    \hspace{0.5em}
    \begin{ytableau}
        \nonumber & \nonumber & x_1 & 1 & 1 \\
        \nonumber & x_2 & 1 & 2\\
        x_3 & 1
    \end{ytableau}
    \hspace{3em}
    $\mathrm{rect}(X\sqcup T)$:
    \hspace{0.5em}
    \begin{ytableau}
        x_1 & 1 & 1 & 1 & 1\\
        x_2 & 2\\
        x_3
    \end{ytableau}
\end{center}
\vspace{1em}
\[
    \begin{tikzpicture}
        \node at (0,0) {
        \begin{ytableau}
            \nonumber & \nonumber & x_1 & 1 & 1 \\
            \nonumber & x_2 & 1 & 2\\
            x_3 & 1
        \end{ytableau}};
        \node at (2.25,.3) {$\hop(e(X),T)$};
        \draw[-{latex}] (1.25,0) -- ++(2,0);
        \node at (4.625,0) {\begin{ytableau}
            \nonumber & \nonumber & 1 & 1 & x_1\\
            \nonumber & 1 & 2 & x_2\\
            1 & x_3
        \end{ytableau}};
        \node at (6.75,.3) {$\shuffle$};
        \draw[-{latex}] (6,0) -- ++(1.625,0);
        \node at (9,0) {\begin{ytableau}
            \nonumber & \nonumber & x_1 & 1 & 1 \\
            \nonumber & x_2 & 1 & 2\\
            x_3 & 1
        \end{ytableau}};
    \end{tikzpicture}
\]
\end{example}

In \cite[Proposition 6.3]{1box_a}, a characterization of fixed points where $|\lambda|=1$ was given which required $\ybox$ to only make hops over adjacent entries. For the general setting, we use the array algorithm defined in Proposition \ref{prop:array algorithm} to analyze the paths of individual boxes without needing to consider the evacuation of $X$. 
However, Example \ref{ex:monodromy} illustrates the fact that, in general, the paths of each $x_i$ need not be connected when $|\lambda|>1$ for our algorithm to produce a fixed point. In fact, such examples can be produced for rectified tableaux as well. Hence, a classification may require a non-trivial extension of the $|\lambda|=1$ statement or a different one entirely.

\begin{remark}
Certain steps of our algorithm generate \emph{genomic tableaux} (see \cite{PECHENIK_2017}), as in the following example. We plan to characterize this connection precisely in a future paper.
\end{remark}

\begin{example}
For this example, we make use of Remark \ref{remark_stagger} to stagger the steps of Phase 1 to generate a genomic tableau. Here we assume we are in the middle of Phase 1 and $x_3$ has switched past a 1, and we are about to stagger the steps of Phase 1 by moving $x_3$ past a 2 and both $x_1$ and $x_2$ past a 1.

\[
\displaystyle{\left(
\begin{tikzpicture}[baseline={(0,.05)}]
        \node at (2.75,0) {\begin{ytableau}
            \nonumber & \nonumber & x_1 & 1\\
            \nonumber & x_2 & x_3\\
            1 & 2
        \end{ytableau}};
        \node at (4.75,.3) {Phase 1};
        \draw[-{latex}] (4,0) -- ++(1.5,0);
        \node at (6.75,0) {\begin{ytableau}
            \nonumber & \nonumber & 1 & 1\\
            \nonumber & x_1 & 2\\
            x_2 & x_3
        \end{ytableau}};
    \end{tikzpicture}\right)
    \begin{tikzpicture}[baseline={(0,.05)}]
        \draw[|-{latex}] (8.25,0) -- ++(1.25,0);
        \node at (10.75,0) {\begin{ytableau}
             \nonumber & \nonumber & 1_1 & 1_2\\
            \nonumber & 1_1 & 2_3\\
            1_1 & 2_3
         \end{ytableau}};
    \end{tikzpicture}.
}
\]
\end{example}

\subsection{Computational complexity}\label{sec:complexity}

We now analyze the run time of our algorithms and their computational advantage.  The default method of computing coswitch on a pair of tableaux has an upper bound on its complexity of roughly $$2|\alpha|(\ell(\beta) +\beta_1+ \ell(\lambda) +\lambda_1) +|\beta|(\ell(\lambda)+\lambda_1).$$ 
Indeed, $\
\ell(\beta)+\beta_1+\ell(\lambda)+\lambda_1$ is an upper bound on the length of the slide path of an element of $\alpha$ through $\beta\sqcup\lambda$ for any JDT slide we perform, since any element moves vertically at most $\ell(\beta)+\ell(\lambda)$ steps and horizontally at most $\beta_1+\lambda_1$ steps. Since we rectify and unrectify $\beta\sqcup\lambda$ into the shape of $\alpha$, we perform $2|\alpha|$ of these slides. When we switch our $\beta$ and $\lambda$ after our first rectification step, $\ell(\lambda)+\lambda_1$ is an upper bound on the length of the slide path of any element during this switching step since each element moves vertically at most $\ell(\lambda)$ steps and horizontally at most $\lambda_1$ steps. We perform a number of slides equal to $|\beta|$, so we multiply this upper bound on our slide path length by $|\beta|$.  Putting this all together yields the formula above.

Our algorithms perform the same computation with an upper bound on their complexity of $$|\beta|(\ell(\lambda)+\ell(\beta))(2\beta_1+2\lambda_1).$$  
Indeed, in Phase 1 (via hopping), suppose for a given element of the tableau of content $\beta$, we perform $N$ Phase 1 hops. We claim that in Phase 2 (via crystal operators), we apply a number of crystal operators that is at most equal to $\ell(\lambda)+\ell(\beta)-N$ corresponding to that element of $\beta$. Indeed, during Phase 1, we can add up to $\ell(\beta)$ new entries to $\lambda$. So, the total number of Phase 1 or Phase 2 operations applied for each element of $\beta$ is at most $\ell(\lambda)+\ell(\beta)$, and the Phase 2 crystal operations take at most $2\beta_1+2\lambda_1$ steps (to compute the pairing on each $(i,i+1)$ subword).  Thus, the total number of computational steps is bounded above by our claimed formula.


The key difference to note between these two upper bounds is that the upper bound for our algorithms does not depend on $\alpha$ in any way, which is a major advantage for large $|\alpha|$, i.e. for skew shapes that are far from rectified.

        





\section{\texorpdfstring{$\alg$}{pesh} in the orthogonal Grassmannian}\label{sec:type B}

We now extend the analogous results of Gillespie, Levinson, and Purbhoo for Schubert curves in the odd orthogonal Grassmannian from \cite{1box_b}. Our work is consistent with their notation and conventions, but we will still review some important details for the background of type B.

The previous sections have been written in such a way that the statements and proofs carry over approximately word-for-word in Type B, relying only on the formal properties of the analogous objects (i.e. coplactic operators, ballotness, JDT, dual equivalence, etc.), so we mostly omit the proofs in this section. The purpose of this section is simply to provide the correct descriptions of the algorithms.

\subsection{Background}

Unlike our work above with the Grassmannian, cohomology of the odd orthogonal Grassmannian is computed by \textbf{shifted Young tableaux}, which are tableaux in the alphabet $\{1'<1<2'<2<\ldots\}$ whose (skew) shapes fit in a height $n$ staircase as shown below:

\[
    \begin{ytableau}
        \nonumber & \nonumber & \nonumber & 2' & 2\\
        \none & \nonumber & 1' & 2' & 4\\
        \none & \none & 1 & 3' &\nonumber\\
        \none & \none & \none & 3 &\nonumber\\
        \none &\none &\none &\none & \nonumber
    \end{ytableau}
\]

A shifted Young tableau is \textbf{semistandard} if the unprimed entries weakly increase along rows and strongly increase down columns, while the primed entries strongly increase along rows and weakly increase down columns. The example above is also semistandard. We again assume all shifted Young tableaux are semistandard unless otherwise stated.

\begin{definition}
Let $w=w_1w_2\ldots $ be a string in symbols $\{1',1,2',2,\ldots \}$. The leftmost $i$ entry (whether primed or unprimed) in $w$ is denoted $\mathrm{first}(i,w)$ or just $\mathrm{first}(i)$ if $w$ is clear from context. The \textbf{canonical form} of $w$ is obtained by replacing $\mathrm{first}(i,w)$ with an unprimed $i$ entry for all $i\in \{1,2,\ldots \}$.
\end{definition}

\begin{definition}
Two strings $w$ and $v$ are \textbf{equivalent} if they have the same canonical form; a \textbf{word} in type B is an equivalence class of strings under this relation.
\end{definition}

\begin{example}
    The canonical form of the word $12'1'121'$ is $121'121'$.
\end{example}

Similarly, two shifted tableaux $X$ and $T$ are equivalent if they have the same shape and their (row) reading words are equivalent. When we define the algorithms, it will be more convenient to start with the tableaux in canonical form, and to switch between representatives, as needed.

In this setting, the \textbf{standardization} of a word is formed by replacing the letters by $1,\ldots,n$ from least to greatest, with ties broken in \textit{reverse} reading order for primed entries and reading order for unprimed entries; standardization order will also change accordingly.

JDT slides have an analogous definition on shifted Young tableaux as discussed by Sagan and Worley \cite{Sagan}\cite{worley}. The sliding rules are mostly the same as before, but we have two special rules for the case where a box moves along the diagonal as shown below.

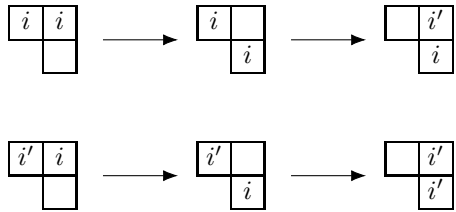
\[
\begin{tikzpicture}
    \node{
    \begin{ytableau}
        i & i\\
        \none & \nonumber
    \end{ytableau}
    };
    \draw[-{Latex[length=2mm]}] (.8,0) -- ++(1,0);
    \node at (2.5,0)
    {
    \begin{ytableau}
        i & \nonumber\\
        \none & i
    \end{ytableau}
    };
    \draw[-{Latex[length=2mm]}] (3.3,0) -- ++(1,0);
    \node at (5,0)
    {
    \begin{ytableau}
        \nonumber & i'\\
        \none & i
    \end{ytableau}
    };
\end{tikzpicture}
\]

\[
\begin{tikzpicture}
    \node{
    \begin{ytableau}
        i' & i\\
        \none & \nonumber
    \end{ytableau}
    };
    \draw[-{Latex[length=2mm]}] (.8,0) -- ++(1,0);
    \node at (2.5,0)
    {
    \begin{ytableau}
        i' & \nonumber\\
        \none & i
    \end{ytableau}
    };
    \draw[-{Latex[length=2mm]}] (3.3,0) -- ++(1,0);
    \node at (5,0)
    {
    \begin{ytableau}
        \nonumber & i'\\
        \none & i'
    \end{ytableau}
    };
\end{tikzpicture}
\]

Using these basic facts, rectification, slide and dual equivalence classes, shifted Littlewood-Richardson tableaux/highest-weight representatives, and the switching operations in Section \ref{sec:switching-ops} are analogously defined for type B.

Furthermore, the geometry discussed in Section \ref{sec:geometry} is also completely analogous to the type B version; each of the switching operations corresponds to a special case of evacuation shuffling on chains of shifted Littlewood-Richardson tableaux.

Ballotness, however, requires the use of lattice walks as described in Section 2.5 of \cite{1box_b}. Gillespie, Levinson, and Purbhoo also defined a set of coplactic operators, $E_i,F_i,E_i',$ and $F_i'$ on shifted skew Young tableaux \cite[Definition 3.3 and 5.4]{typeb_crystals}, which have many of the same formal properties as the type A crystal operators $E_i, F_i$. While these are relevant to the statement and computation of the hopping algorithm, the specifics are not needed for the proofs so we omit their definitions.


\begin{definition}
    Let $T$ be a shifted Littlewood-Richardson tableau. Let $x$ be an entry of a representative of $T$ (not necessarily the canonical form of $T$). Let $i$ be an entry such that no other $i$ symbols lie between $x$ and $i$ in reading order, where we distinguish primed and unprimed symbols. A \textbf{switch} refers to swapping the positions of $x$ and $i$. A switch is \textbf{valid} if $T$ omitting $x$ remains Littlewood-Richardson after the switch. A valid switch is called a \textbf{hop across} an $i$ if $i$ lies before $x$ in standardization order prior to the switch, and an \textbf{inverse hop} otherwise.
\end{definition}

Note that some switches depend on our choice of representative for the tableau and its reading word. For example, the figure on the left shows an inverse hop across a $1'$, while the right is not even a switch.
\[
1'1x\rightarrow x11' \hspace{1in} 11x\rightarrow x11
\]

Finally, we use the notation $(i)^*$ to denote a (possibly empty) subword of arbitrary length comprised entirely of $i$ entries consecutive in reading order.

\subsection{Local shifted algorithms and related results}

We now define a local version of shifted $\alg$ using local moves which resemble shifted JDT.

\begin{definition}\label{def:type-b-hop}
(Shifted hopping algorithm). Let $X$ be a shifted semistandard Young tableau such that $|X|=n$, and let $T$ be a shifted Littlewood-Richardson tableau with largest entry d. Then the shifted hopping algorithm is defined as follows:

Let $x_i$ denote the $i$-th smallest entry of $X$ in standardization order. We first canonicalize $T$ and then begin the algorithm in Phase 1 with $i=n$ and $j=1$.

\begin{itemize}
    \item \textbf{Phase 1:}
    If there is a $j'$ after $x_i$ in reading order, hop $x_i$ across a $j'$, then hop across a $j$. Increment $j$ by 1 and repeat Phase 1. If there is no such $j'$, go to Transition Phase 1.

    \item \textbf{Transition Phase 1:} Record $a_{n-i+1}=j$.\\
    If $i=1$, replace $x_1$ with $x_n$ and go to Phase 2 with the current $i$ and $j$.\\
    Otherwise, replace $x_i$ with $j'$. Decrement $i$ by 1 and set $j=1$. Canonicalize the resulting tableau and then repeat Phase 1.

    \item \textbf{Phase 2:} If $x_{n-i+1}$ most recently switched with an earlier entry in reading order, or $x_{n-i+1}$ has not yet moved, enter Phase 2(a) below; otherwise, skip to Phase 2(b).
    \begin{itemize}
        \item \textbf{Phase 2(a):} If $x_{n-i+1}$ precedes all $j$ and $j'$ entries in reading order, or there are no $j$ or $j'$ entries, go to Phase 2(b).\\
        Otherwise, change $\mathrm{first}(j)$ to $j'$, then perform as many valid inverse hops across $j'$ as possible.\\ 
        If $x_{n-i+1}$ precedes all $j$ and $j'$ entries in reading order, go to Phase 2(b).\\ 
        Otherwise, if the $j,j+1$-reading word has the form $\ldots j(j+1)^*x_{n-i+1}\ldots $, hop $x_{n-i+1}$ across the $j$.\\ 
        Increment $j$ by 1 and repeat Phase 2.
        \item \textbf{Phase 2(b):} Perform as many valid inverse hops across $j$ as possible. Then, if the $j,j+1$-reading word has the form $\ldots x_{n-i+1}(j')^*(j+1')\ldots $, hop across the $j+1'$.\\
        If $j<d+n-i+1$, increment $j$ by 1 and repeat Phase 2.\\
        Otherwise, go to Transition Phase 2.
    \end{itemize}
    \item \textbf{Transition Phase 2:} If $i=n$, the algorithm is complete.\\
    Otherwise, increment $i$ by 1 and set $j=a_{n-i+1}$. Replace the smallest $j$ entry in standardization order with $x_{n+1-i}$. Canonicalize the remaining tableau and repeat Phase 2.
\end{itemize}
\end{definition}

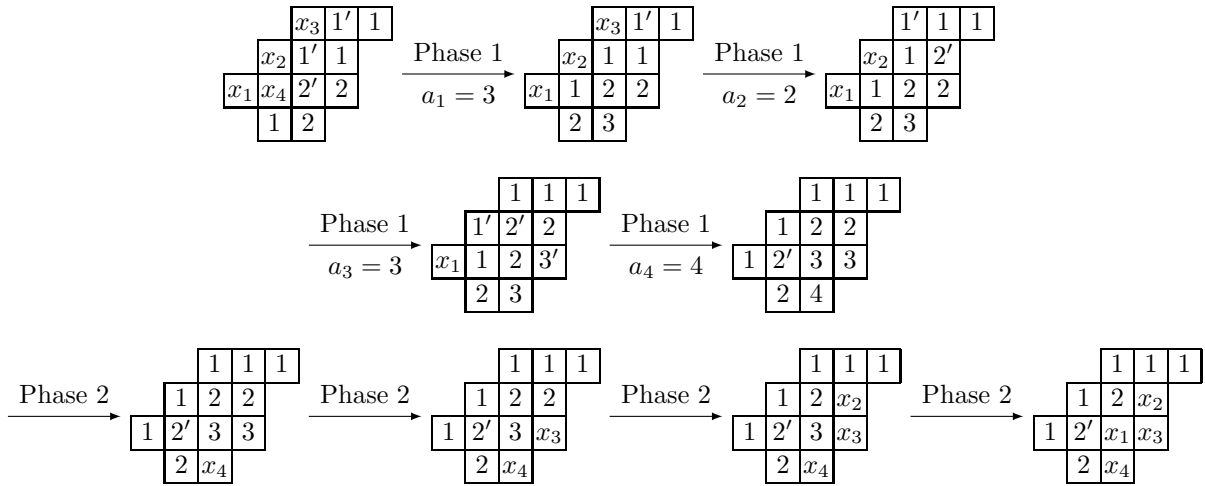
\begin{example}
    \[
    \begin{tikzpicture}
        \node at (4.25,0) {
        \begin{ytableau}
            \none & \none & x_3 & 1' & 1\\
            \none & x_2 & 1' & 1\\
            x_1 & x_4 & 2' & 2\\
            \none & 1 & 2
        \end{ytableau}};
        \node at (6.25,.3) {Phase 1};
        \node at (6.25,-.3) {$a_1=3$};
        \draw[-{latex}] (5.5,0) -- ++(1.5,0);
        \node at (8.25,0) {
        \begin{ytableau}
            \none & \none & x_3 & 1' & 1\\
            \none & x_2 & 1 & 1\\
            x_1 & 1 & 2 & 2\\
            \none & 2 & 3
        \end{ytableau}};
        \node at (10.25,.3) {Phase 1};
        \node at (10.25,-.3) {$a_2=2$};
        \draw[-{latex}] (9.5,0) -- ++(1.5,0);
        \node at (12.25,0) {
        \begin{ytableau}
            \none & \none & 1' & 1 & 1\\
            \none & x_2 & 1 & 2'\\
            x_1 & 1 & 2 & 2\\
            \none & 2 & 3
        \end{ytableau}};
    \end{tikzpicture}
    \]
    \[
    \begin{tikzpicture}
        \node at (-2,.3) {Phase 1};
        \node at (-2,-.3) {$a_3=3$};
        \draw[-{latex}] (-2.75,0) -- ++(1.5,0);
        \node at (0,0) {
        \begin{ytableau}
            \none & \none & 1 & 1 & 1\\
            \none & 1' & 2' & 2\\
            x_1 & 1 & 2 & 3'\\
            \none & 2 & 3
        \end{ytableau}};
        \node at (2,.3) {Phase 1};
        \node at (2,-.3) {$a_4=4$};
        \draw[-{latex}] (1.25,0) -- ++(1.5,0);
        \node at (4,0) {
        \begin{ytableau}
            \none & \none & 1 & 1 & 1\\
            \none & 1 & 2 & 2\\
            1 & 2' & 3 & 3\\
            \none & 2 & 4
        \end{ytableau}};
    \end{tikzpicture}
    \]
    \[
    \begin{tikzpicture}
        \node at (-2,.3) {Phase 2};
        \draw[-{latex}] (-2.75,0) -- ++(1.5,0);
        \node at (0,0) {
        \begin{ytableau}
            \none & \none & 1 & 1 & 1\\
            \none & 1 & 2 & 2\\
            1 & 2' & 3 & 3\\
            \none & 2 & x_4
        \end{ytableau}};
        \node at (2,.3) {Phase 2};
        \draw[-{latex}] (1.25,0) -- ++(1.5,0);
        \node at (4,0) {
        \begin{ytableau}
            \none & \none & 1 & 1 & 1\\
            \none & 1 & 2 & 2\\
            1 & 2' & 3 & x_3\\
            \none & 2 & x_4
        \end{ytableau}};
        \node at (6,.3) {Phase 2};
        \draw[-{latex}] (5.25,0) -- ++(1.5,0);
        \node at (8,0) {
        \begin{ytableau}
            \none & \none & 1 & 1 & 1\\
            \none & 1 & 2 & x_2\\
            1 & 2' & 3 & x_3\\
            \none & 2 & x_4
        \end{ytableau}};
        \node at (10,.3) {Phase 2};
        \draw[-{latex}] (9.25,0) -- ++(1.5,0);
        \node at (12,0) {
        \begin{ytableau}
            \none & \none & 1 & 1 & 1\\
            \none & 1 & 2 & x_2\\
            1 & 2' & x_1 & x_3\\
            \none & 2 & x_4
        \end{ytableau}};
    \end{tikzpicture}
    \]
\end{example}

While the computation of the type B hopping algorithm is very different from that of the Type A version, the overall structure has some notable similarities: the conditional statement in Phase 1 detects the rectification shape of the $(j-1,j)$-subword, and moves us to a Transition Phase if the rectification shape is a single row. In the rectified setting specifically, the transition data also determines the row that each entry of $X$ lies on just before Phase 2, and Phase 2 is equivalent to sliding the corresponding box to the end of its row. These similarities are more than superficial as the following statements and proofs are nearly identical to the type A analogues:

\begin{lemma}
    The shifted hopping algorithm is coplactic.
\end{lemma}

\begin{proof}
    The switches and relabeling steps were shown to be coplactic in \cite{1box_b}.
\end{proof}

\begin{lemma}{\label{lemma:always-ballot-B}}
    Let $X$ be a $\SSYT$ and $T$ be a $\mathrm{LR}$ tableau extending $X$. Let $T'$, including all entries $X$, be a tableau that appears after some step of the computation of $\hop(X,T)$. Then:

\begin{enumerate}
    \item $T'$ is LR when omitting every entry of $X$ that has not reached Phase 1b.
    \item $T'$ is semistandard when omitting every entry of $X$ that has not reached Phase 1b. 
    \item After $x_i$ performs the $j$-th step of Phase 1 or Phase 2, it is an inner corner of the sub-tableau $T'_{>j}$ containing only those entries greater than $i$.
\end{enumerate}
\end{lemma}

\begin{proof}
    The proof for Lemma \ref{lemma:always-ballot-B} is essentially the same as that of Lemma \ref{alwaysballot}, except we refer to \cite[Theorem 5.7]{1box_b} for the one box case.
\end{proof}

\begin{theorem}\label{thm:type-b-hop=coswitch}
    The output of $\coswitch(X,T)$ and the output of the shifted hopping algorithm agree.
\end{theorem}

\begin{proof}
    Once again, the proof is identical to that of Theorem \ref{thm: hop=esh}. The only modification required is that the hops are shown to be coplactic by \cite[Theorem 5.19 and Theorem 5.35]{1box_b}.
\end{proof}

The similarities discussed above become more apparent in the shifted coplactic algorithm, whose statement strongly resembles that of the type A crystal algorithm. Though, we note that the actual computations are still very different as we are using the type B coplactic operators instead of type A crystal operators here.

\begin{definition}\label{def:type-b-crystal}
    (Shifted coplactic algorithm). Let $X$ be a shifted standard Young tableau such that $|X|=n$, and let $T$ be a shifted Littlewood-Richardson tableau with largest entry $d$. Then the shifted coplactic algorithm is defined as follows:

    Replace the entries of $X$ in standardization order with $-(n-1),\ldots,0$. Set $j=1$ and $i=n$.
\end{definition}

\begin{itemize}
    \item \textbf{Phase 1:} If $E'_{j-1}(T)\neq E_{j-1}(T)$, apply $E'_{j-1}\circ E^{\weight(T)_j-2}_{j-1}$. Increment $j$ by 1 and repeat Phase 1.\\
    Otherwise, go to Transition Phase.
    \item  \textbf{Transition Phase:} Replace the unique $j-1$ with $j'$ and add 1 to all values less than $j-1$. Record $a_{n-i+1}=j$. \\
    If $i=1$, go to Phase 2.\\
    Otherwise, decrement $i$ by 1 and set $j=1$. Then, repeat Phase 1.
    \item \textbf{Phase 2:} Apply the composition of operators
    \begin{center}
        $F_{d+n-i}\circ \ldots \circ F_{a_{n-i+1}+1}\circ F_{a_{n-i+1}}$
    \end{center}
    to the resulting tableau.\\
    If $i>1$, decrement $i$ by 1 and repeat Phase 2.\\
    Otherwise, replace the new $d+1,\ldots,d+n$ entries with $x_1,\ldots, x_n$ respectively.
\end{itemize}

We demonstrate the shifted coplactic algorithm using the same example as before:

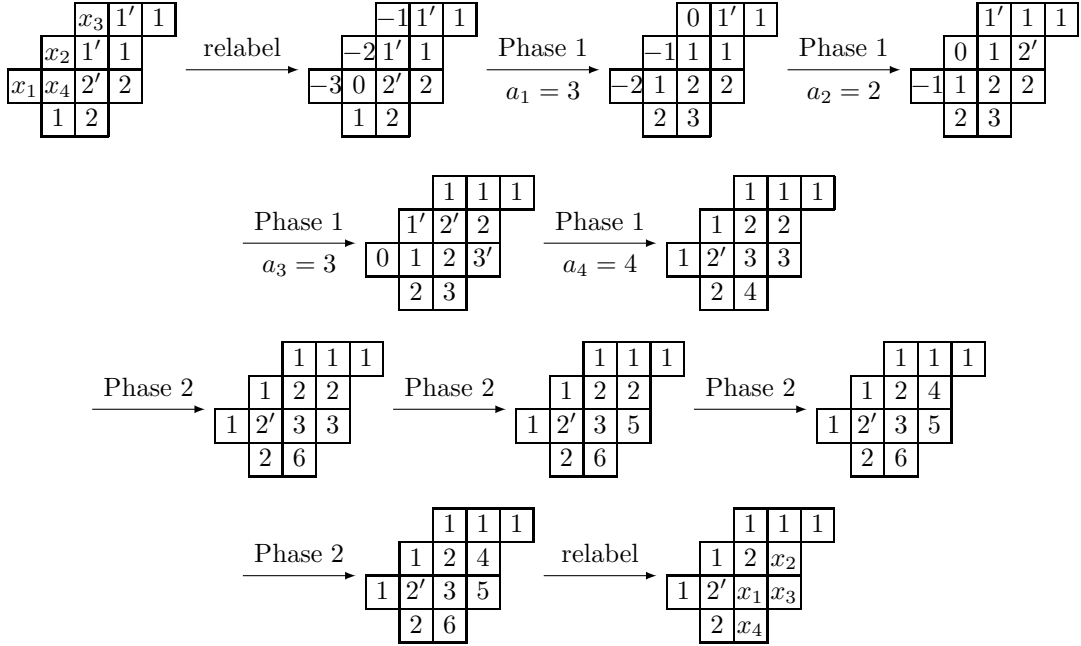
\begin{example}
    \[
    \begin{tikzpicture}
        \node at (4.25,0) {
        \begin{ytableau}
            \none & \none & x_3 & 1' & 1\\
            \none & x_2 & 1' & 1\\
            x_1 & x_4 & 2' & 2\\
            \none & 1 & 2
        \end{ytableau}};
        \node at (6.25,.3) {relabel};
        \draw[-{latex}] (5.5,0) -- ++(1.5,0);
        \node at (8.25,0) {
        \begin{ytableau}
            \none & \none & -1 & 1' & 1\\
            \none & -2 & 1' & 1\\
            -3 & 0 & 2' & 2\\
            \none & 1 & 2
        \end{ytableau}};
        \node at (10.25,.3) {Phase 1};
        \node at (10.25,-.3) {$a_1=3$};
        \draw[-{latex}] (9.5,0) -- ++(1.5,0);
        \node at (12.25,0) {
        \begin{ytableau}
            \none & \none & 0 & 1' & 1\\
            \none & -1 & 1 & 1\\
            -2 & 1 & 2 & 2\\
            \none & 2 & 3
        \end{ytableau}};
        \node at (14.25,.3) {Phase 1};
        \node at (14.25,-.3) {$a_2=2$};
        \draw[-{latex}] (13.5,0) -- ++(1.5,0);
        \node at (16.25,0) {
         \begin{ytableau}
            \none & \none & 1' & 1 & 1\\
            \none & 0 & 1 & 2'\\
            -1 & 1 & 2 & 2\\
            \none & 2 & 3
        \end{ytableau}};
    \end{tikzpicture}
    \]
    \[
    \begin{tikzpicture}
        \node at (2,.3) {Phase 1};
        \node at (2,-.3) {$a_3=3$};
        \draw[-{latex}] (1.25,0) -- ++(1.5,0);
        \node at (4,0) {
        \begin{ytableau}
            \none & \none & 1 & 1 & 1\\
            \none & 1' & 2' & 2\\
            0 & 1 & 2 & 3'\\
            \none & 2 & 3
        \end{ytableau}};
        \node at (6,.3) {Phase 1};
        \node at (6,-.3) {$a_4=4$};
        \draw[-{latex}] (5.25,0) -- ++(1.5,0);
        \node at (8,0) {
        \begin{ytableau}
            \none & \none & 1 & 1 & 1\\
            \none & 1 & 2 & 2\\
            1 & 2' & 3 & 3\\
            \none & 2 & 4
        \end{ytableau}};
    \end{tikzpicture}
    \]
    \[
    \begin{tikzpicture}
        \node at (-2,.3) {Phase 2};
        \draw[-{latex}] (-2.75,0) -- ++(1.5,0);
        \node at (0,0) {
        \begin{ytableau}
            \none & \none & 1 & 1 & 1\\
            \none & 1 & 2 & 2\\
            1 & 2' & 3 & 3\\
            \none & 2 & 6
        \end{ytableau}};
        \node at (2,.3) {Phase 2};
        \draw[-{latex}] (1.25,0) -- ++(1.5,0);
        \node at (4,0) {
        \begin{ytableau}
            \none & \none & 1 & 1 & 1\\
            \none & 1 & 2 & 2\\
            1 & 2' & 3 & 5\\
            \none & 2 & 6
        \end{ytableau}};
        \node at (6,.3) {Phase 2};
        \draw[-{latex}] (5.25,0) -- ++(1.5,0);
        \node at (8,0) {
        \begin{ytableau}
            \none & \none & 1 & 1 & 1\\
            \none & 1 & 2 & 4\\
            1 & 2' & 3 & 5\\
            \none & 2 & 6
        \end{ytableau}};
    \end{tikzpicture}
    \]
    \[
    \begin{tikzpicture}
        \node at (-2,.3) {Phase 2};
        \draw[-{latex}] (-2.75,0) -- ++(1.5,0);
        \node at (0,0) {
        \begin{ytableau}
            \none & \none & 1 & 1 & 1\\
            \none & 1 & 2 & 4\\
            1 & 2' & 3 & 5\\
            \none & 2 & 6
        \end{ytableau}};
        \node at (2,.3) {relabel};
        \draw[-{latex}] (1.25,0) -- ++(1.5,0);
        \node at (4,0) {
        \begin{ytableau}
            \none & \none & 1 & 1 & 1\\
            \none & 1 & 2 & x_2\\
            1 & 2' & x_1 & x_3\\
            \none & 2 & x_4
        \end{ytableau}};
    \end{tikzpicture}
    \]
\end{example}

\begin{theorem}\label{thm:type-b-hop=crystal}
    Let $X$ be a shifted standard Young Tableau and $T$ be a shifted Littlewood-Richardson tableau. Then the outputs of the shifted coplactic and hopping algorithms agree.
\end{theorem}

\begin{proof}
    The relabeling steps of the local algorithms are effectively the same; we simply use $-(n-1),\ldots, 0$ in the coplactic algorithm for the same $y_1,\ldots,y_n$ entries in the hopping algorithm. Then, the fact that Phases 1 and 2 of the shifted hopping algorithm and the shifted coplactic algorithm agree follows from Theorems 5.19 and 5.35 respectively of \cite{1box_b}.
\end{proof}

This gives us the analogous statements of Definitions \ref{def:type_a_hop} and \ref{def:type-A-crystal-alg}. The remaining results from Section \ref{sec:related-results} may also be extended to the shifted algorithms without modifying the proofs. In particular, we obtain local algorithms which compute shifted $\coplacticEvac$ and shifted $\alg$ by omitting the initial evacuation step of either local algorithm.

We also give an example of a fixed point under monodromy. This example also has the property that some boxes perform non-adjacent hops in both Phase 1 and Phase 2.

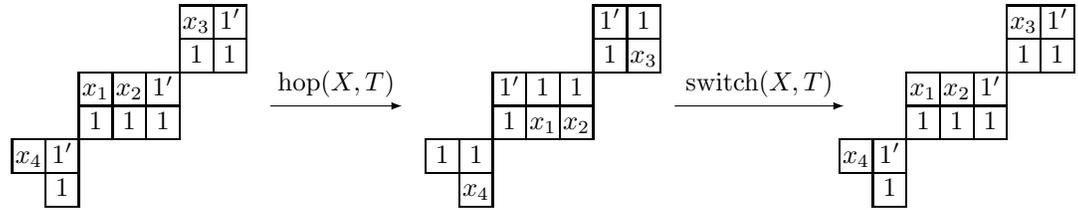
\begin{example}
\[
\begin{tikzpicture}
    \node at (0,0){
    \begin{ytableau}
        \none &\none &\none &\none &\none & x_3 & 1'\\
        \none &\none &\none &\none &\none & 1 & 1\\
        \none &\none & x_1 & x_2 & 1'\\
        \none &\none & 1 & 1 & 1\\
        x_4 & 1' \\
        \none & 1
    \end{ytableau}};
    \node at (2.75,.3) {$\hop(X,T)$};
    \draw[-{latex}] (1.875,0) -- ++(1.75,0);
    \node at (5.5,0){
    \begin{ytableau}
        \none &\none &\none &\none &\none & 1' & 1\\
        \none &\none &\none &\none &\none & 1 & x_3\\
        \none &\none & 1' & 1 & 1\\
        \none &\none & 1 & x_1 & x_2\\
        1 & 1 \\
        \none & x_4
    \end{ytableau}};
    \node at (8.375,.3) {$\shuffle(X,T)$};
    \draw[-{latex}] (7.25,0) -- ++(2.25,0);
    \node at (11,0){
    \begin{ytableau}
        \none &\none &\none &\none &\none & x_3 & 1'\\
        \none &\none &\none &\none &\none & 1 & 1\\
        \none &\none & x_1 & x_2 & 1'\\
        \none &\none & 1 & 1 & 1\\
        x_4 & 1' \\
        \none & 1
    \end{ytableau}};
\end{tikzpicture}
\]
\end{example}

Finally, we state the reverse of each algorithm. The coplactic statement of local $\alg{}$ is fairly straightforward to invert since $E_i$ and $F_i$ are partial inverses, the only difficulty is detecting when an entry finishes Phase 2:

\begin{definition}
    (Inverse shifted coplactic algorithm). Let $T$ be a shifted Littlewood-Richardson tableau, and let $X$ be a shifted standard tableau of size $n$ extending $T$. Let $x_i$ be the $i$-th smallest entry of $X$ in standardization order, and let $d$ be the largest entry of $T$. Then the inverse shifted coplactic algorithm is defined as follows:

    Replace each $x_i$ with $d+i$. Begin the algorithm in Phase 1 with $i=1$.

    \begin{itemize}
        \item \textbf{Reverse Phase 2}: Apply the coplactic operators $E_{d+i-1},E_{d+i-2},\ldots$ until the first index $j$ such that $E_{j-1}$ is undefined. Then, record $a_i=j$.\\
        If $i<n$, increment $i$ by 1 and repeat Reverse Phase 2.\\
        Otherwise, go to Reverse Phase 1 with $i=n$.
        \item \textbf{Reverse Phase 1}: Subtract 1 from all entries with value less than $a_i$. Replace the smallest $a_i$ in standardization order with $a_i-1$. Apply the following sequence of coplactic operators to the resulting tableau:
        \[
        F_0^{wt(T)_0-2}\circ F'_0\circ F_1^{wt(T)_1-2}\circ F'_1\circ \ldots \circ F_{a_i-2}^{wt(T)_{a_i-2}-2}\circ F'_{a_i-2}.
        \]

        Then, replace the unique 0 entry with $x_{n-i+1}$.

        If $i=1$, the algorithm is complete.

        Otherwise, decrement $i$ by 1 and repeat Reverse Phase 1
    \end{itemize}
\end{definition}

The hopping algorithm, however, requires specific choices of representatives which do not always match those used in the forward algorithm (Definition \ref{def:type-b-hop}). We now state the inverse hopping algorithm below.

\begin{definition}
    (Inverse shifted hopping algorithm). Let $T$ be a shifted Littlewood-Richardson tableau, and let $X$ be a shifted standard tableau of size $n$ extending $T$. Then the inverse of the shifted hopping algorithm can be computed as follows:

    Begin with $T$ in canonical form. Let $d$ be the \emph{current} largest entry of the tableau, note that $d$ will possibly change as we perform steps of the inverse algorithm. Set $j=d$ and $i=1$. If the reading word formed by replacing $x_i$ with $d+1$ is reverse-ballot, go to Reverse Transition Phase 2; otherwise, begin in reverse Phase 2(b) below.

    \begin{itemize}
        \item \textbf{Reverse Phase 2}: If $x_i$ precedes all $j$ and $j'$ entries in reading order, change $\mathrm{first}(j)$ to $j'$.\\
        If $x_i$ most recently switched with an earlier entry in reading order, or $x_i$ has not yet moved, enter Reverse Phase 2(b) below; otherwise, go to Reverse Phase 2(a).

        \begin{itemize}
        \item \textbf{Reverse Phase 2(b)}: Have $x_i$ perform as many valid hops across $j$ as possible.\\
        If $x_i$ now precedes all $j$ and $j'$ entries in reading order, go to Reverse Phase 2(a).\\
        If replacing $x_i$ by $j$ results in a reverse-ballot $j-1,j$-subword, go to Reverse Transition Phase 2.\\
        Otherwise, if the $j-1,j$-reading word has the form $\ldots (j')(j-1')^*x_i\ldots$, inverse hop $x_i$ across the $j'$ entry.\\
        Decrement $j$ by 1 and repeat Reverse Phase 2.

        \item \textbf{Reverse Phase 2(a)}: Perform as many valid hops across $j'$ as possible.\\
        If replacing $x_i$ by $j'$ results in a reverse-ballot $j-1,j$-subword, go to Reverse Transition Phase 2.\\  
        Otherwise, if the $j-1,j$-reading word has the form $\ldots x_i(j)^*(j-1)\ldots$, inverse hop $x_i$ across the $j-1$ entry.\\
        Decrement $j$ by $1$ and repeat Reverse Phase 2.

        \item \textbf{Reverse Transition Phase 2:} If $i=n$, replace $x_n$ with $x_1$. Decrement $j$ by $1$ and go to Reverse Phase 1.\\
        Otherwise, replace $x_i$ with $j'$ and record $a_{n-i+1}=j$. Increment $i$ by 1 and set $j=d$ where $d$ is the current largest entry. Canonicalize the resulting tableau.\\
        If the word formed by replacing $x_i$ with $d+1$ is reverse-ballot, repeat Reverse Transition Phase 2.\\
        Otherwise, Go to Reverse Phase 2(b).
        \end{itemize}

        \item \textbf{Reverse Phase 1}: 
        If $j=0$, go to Reverse Transition Phase 1.\\
        Otherwise, change $\mathrm{first}(j)$ to $j'$. Inverse hop $x_{n+i-1}$ across a $j$, then inverse hop across a $j'$. Decrement $j$ by 1, then repeat Reverse Phase 1.

        \item  \textbf{Reverse Transition Phase 1:}
        If $i=1$, the algorithm is complete.\\
        Otherwise, decrement $i$ by 1 and set $j=a_{n-i+1}-1$. Replace the smallest $a_{n-i+1}$ in standardization order with $x_{n+i-1}$. Canonicalize the resulting tableau, then go to Reverse Phase 1.
    \end{itemize}
\end{definition}

\begin{theorem}
    The reverse algorithms compute the inverse of those in Definitions \ref{def:type-b-hop} and \ref{def:type-b-crystal}.
\end{theorem}
\begin{proof}
    It is straightforward to see that each step inverts the corresponding step of the shifted hopping and coplactic algorithms respectively and that the transition from Reverse Phase 2 to Reverse Phase 1 corresponds to the transition step in the forward algorithm.
\end{proof}

\bibliographystyle{plain}
\bibliography{refs}

\end{document}